\documentclass[a4paper,11pt, twoside]{article}
\usepackage[T1]{fontenc}
\usepackage[english]{babel}
\pdfoutput=1 

\usepackage{amsmath, amssymb, amsthm, amsfonts, mathtools, mathrsfs}	
\usepackage{cancel}
\usepackage{enumerate}
\usepackage[shortlabels]{enumitem}
\usepackage{booktabs, caption, graphicx, subfig, wrapfig} 	
\usepackage{geometry}
\usepackage{microtype}
\usepackage{xcolor}
\usepackage{bbm}

\usepackage{cite} 

\usepackage{fancyhdr} 

\newcommand{\R}{\mathbb{R}}
\newcommand{\N}{\mathbb{N}}

\numberwithin{equation}{section}

\theoremstyle{plain}
\newtheorem{theorem}{Theorem}[section] 
\newtheorem{proposition}[theorem]{Proposition} 
\newtheorem{lemma}[theorem]{Lemma}

\theoremstyle{definition}
\newtheorem{definition}[theorem]{Definition}
\newtheorem{remark}[theorem]{Remark}

\DeclarePairedDelimiter{\abs}{\lvert}{\rvert}
\DeclarePairedDelimiter{\norm}{\lVert}{\rVert}
\DeclarePairedDelimiter{\duality}{\langle}{\rangle}

\DeclareMathOperator{\divergence}{div}
\DeclareMathOperator{\Span}{span}

\renewcommand{\div}{\divergence}
\newcommand{\eps}{\varepsilon}
\renewcommand{\phi}{\varphi}
\renewcommand{\bar}{\overline}
\renewcommand{\vec}{\boldsymbol}

\def\genspazio #1#2#3#4#5{#1^{#2}(#5,#4;#3)}
\def\spazio #1#2#3{\genspazio {#1}{#2}{#3}T0}

\def\LT {\spazio L}
\def\HT {\spazio H}
\def\WT {\spazio W}
\def\CT #1#2{C^{#1}([0,T];#2)}

\def\Lx #1{L^{#1}(\Omega)}
\def\Lt #1{L^{#1}(0,T)}
\def\Lqt #1{L^{#1}(Q_T)}
\def\Hx #1{H^{#1}(\Omega)}
\def\Wx #1{W^{#1}(\Omega)}

\def\Accorpa #1#2 #3 {\gdef #1{\eqref{#2}--\eqref{#3}}%
	\wlog{}\wlog{\string #1 -> #2 - #3}\wlog{}}

\DeclareFontFamily{U}{mathx}{}
\DeclareFontShape{U}{mathx}{m}{n}{<-> mathx10}{}
\DeclareSymbolFont{mathx}{U}{mathx}{m}{n}
\DeclareMathAccent{\widehat}{0}{mathx}{"70}
\DeclareMathAccent{\widecheck}{0}{mathx}{"71}

\def\mezzo {\frac{1}{2}}
\def\ddt {\frac{\de}{\de t}}
\def\de {\mathrm{d}}

\def\weakstar {\stackrel{\ast}{\rightharpoonup}}
\def\weak {\rightharpoonup}

\def\n {\vec{n}}
\def\phid {\phi_\delta}
\def\mud {\mu_\delta}
\def\sigmad {\sigma_\delta}
\def\dg {\frac{\delta}{\gamma}}
\def\Ncal {\mathcal{N}}
\def\phide {\phi_{\delta, \eps}}
\def\mude {\mu_{\delta, \eps}}
\def\sigmade {\sigma_{\delta, \eps}}
\def\Psipe {\Psi_{+, \eps}}
\def\phiden {\phi^n_{\delta, \eps}}
\def\muden {\mu^n_{\delta, \eps}}
\def\sigmaden {\sigma^n_{\delta, \eps}}

\usepackage[pdfencoding=unicode]{hyperref}
\usepackage{orcidlink}

\begin{document}
	
	\begin{center}
		
		\LARGE{\textbf{On a relaxed Cahn--Hilliard tumour growth model \\ with single-well potential and degenerate mobility}}
		
		\vskip0.75cm
        \Large{\it In memoriam Federica Bubba (1992-2020)}
        
         \vskip0.75cm

		\large{\textsc{Cecilia Cavaterra$^1$} \orcidlink{0000-0002-2754-7714}} \\
		\normalsize{e-mail: \texttt{cecilia.cavaterra@unimi.it}} \\
		\vskip0.4cm
		
		\large{\textsc{Matteo Fornoni$^2$} \orcidlink{0000-0002-9787-047X}} \\
		\normalsize{e-mail: \texttt{matteo.fornoni@unimi.it}} \\
		\vskip0.4cm
		
		\large{\textsc{Maurizio Grasselli$^3$} \orcidlink{0000-0003-2521-2926}} \\
		\normalsize{e-mail: \texttt{maurizio.grasselli@polimi.it}} \\
		\vskip0.4cm

		\large{\textsc{Benoît Perthame$^4$} \orcidlink{0000-0002-7091-1200}} \\
		\normalsize{e-mail: \texttt{Benoit.Perthame@sorbonne-universite.fr}} \\
		\vskip0.4cm
		
		\footnotesize{$^1$Dipartimento di Matematica ``F. Enriques'', Universit\`{a} degli Studi di Milano \& IMATI--C.N.R., Pavia, Italy}
		\vskip0.3cm
		
		\footnotesize{$^2$Dipartimento di Matematica ``F. Enriques'', Universit\`{a} degli Studi di Milano, Italy}
		\vskip0.3cm

		\footnotesize{$^3$Dipartimento di Matematica, Politecnico di Milano, Italy}
		\vskip0.3cm

		\footnotesize{$^4$Sorbonne Université, CNRS, Université Paris-Cité, Inria,  Laboratoire ``Jacques-Louis Lions'', Paris, France}
		\vskip0.5cm
		
	\end{center}

	\begin{abstract}\noindent
		We consider a phase-field system modelling solid tumour growth. This system consists of a Cahn--Hilliard equation coupled with a nutrient equation. The former is characterised by a degenerate mobility and a singular potential. Both equations are subject to suitable reaction terms which model proliferation and nutrient consumption. Chemotactic effects are also taken into account. Adding an elliptic regularisation, depending on a relaxation parameter $\delta>0$, in the equation for the chemical potential, we prove the existence of a weak solution to an initial and boundary value problem for the relaxed system. Then, we let $\delta$ go to zero, and we recover the existence of a weak solution to the original system.
		\vskip3mm
		
		
		\vskip3mm
		
		
	\end{abstract}

	\pagestyle{fancy}
	\fancyhf{}	
	\fancyhead[EL]{\thepage}
	\fancyhead[ER]{\textsc{Cavaterra -- Fornoni -- Grasselli -- Perthame}} 
	\fancyhead[OL]{\textsc{On a relaxed degenerate Cahn--Hilliard tumour growth model}} 
	\fancyhead[OR]{\thepage}

	\renewcommand{\headrulewidth}{0pt}
	\setlength{\headheight}{5mm}

	\thispagestyle{empty} 

	\section{Introduction}

    The diffuse interface or phase-field approach has recently found many applications in the description of biological phenomena, among which tumour growth is a prime example \cite{OHP2010, cristini:lowengrub}. 
    This is due to a solid physical foundation on continuum mechanics and the theory of mixtures, and to the mathematical properties of the corresponding systems of PDEs, both from the analytical and numerical point of view.
	In this context, a basic phase-field model in the mechanics of living tissues is the following initial and boundary value problem for a degenerate Cahn--Hilliard equation
	\begin{alignat*}{2}
		& \partial_t \phi - \div (b(\phi) \nabla \mu) = 0 \quad && \hbox{in $Q_T$,} \\
		& \mu = - \gamma \Delta \phi + \Psi'(\phi) \quad && \hbox{in $Q_T$,} \\
		& \partial_{\n} \phi = b(\phi) \partial_{\n} \mu = 0 \quad && \hbox{on $\Sigma_T$,} \\
		& \phi(0) = \phi_0 \quad && \hbox{in $\Omega$.}
	\end{alignat*}
	Here $Q_T = \Omega \times (0,T)$, where $T>0$ is a given final time, and $\Sigma_T = \partial \Omega \times (0,T)$, where $\Omega \subset \R^d$, $d = 2,3$, is a given bounded domain.
	In the above model, $\phi$ usually represents the volume fraction of one of the two cell types, while $\mu$ is the corresponding chemical potential, namely the first variation of the Ginzburg--Landau-type free energy
	\[
		\mathcal{E}_{GL}(\phi) = \int_\Omega \frac{\gamma}{2} \abs{\nabla \phi}^2 + \Psi(\phi) \, \de x.
	\]
	In this case, $\gamma > 0$ is a positive parameter related to the width of the diffuse interface, that is, the region of width proportional to $\sqrt{\gamma}$ where $\phi$ varies continuously between the pure phases $\phi = 0$ and $\phi = 1$.
	Thus, the first term in the functional penalises steep transitions between the two phases, while the potential $\Psi(\phi)$ is typically of double-well type with equal minima in $\phi=0$ and $\phi=1$. However, in solid tumour growth modelling (see, e.g., \cite{AACG2017,ACGAC2018,BP2003} and references therein), a more appropriate assumption for  $\Psi$ is that it is concave and degenerate near $\phi=0$ (short-range attraction) and convex for  $\phi$ close to $1$ (long-range repulsion). Moreover, it is usually taken singular at $\phi=1$ to model saturation by one phase.
	A prototypical potential density satisfying these properties is the \emph{single-well logarithmic potential} of Lennard-Jones type
	\begin{equation}
		\label{eq:singlewell}
		\Psi(s) = - (1 - s^*) \log(1 - s) - \frac{s^3}{3} - (1 - s^*)\frac{s^2}{2} - (1 - s^*)s \quad \hbox{for $s \in [0,1)$},
	\end{equation}
	for some $s^* \in (0,1)$, corresponding to the position of the single-well.
	The mobility function $b(\cdot)$ is usually assumed to degenerate at pure phases, such as
	\begin{equation}
		\label{eq:mobility}
		b(s) = s(1 - s)^2,
	\end{equation}
	thus formally enforcing the property that $0 \le \phi \le 1$.

	The combination of a singular potential and a degenerate mobility makes the above equation very challenging to study analytically.
	Indeed, only very weak existence results are available in the literature \cite{AACG2017,ABCG2025} (see also \cite{EG1996} for the classical Cahn--Hilliard equation, and the recent \cite{GKS2025} for an anisotropic version).
	At the same time, numerical analysis is also highly non-trivial, due to the degeneracy of the fourth-order differential operator.
	For these reasons, in \cite{PP2021}, the authors introduced a relaxed degenerate Cahn--Hilliard equation combined with a convex splitting. The relaxation is obtained by adding an elliptic regularisation term to the definition of the chemical potential~$\mu$, multiplied by a small parameter $\delta > 0$.
	In this way, the resulting Cahn--Hilliard equation can be seen as a system of two coupled PDEs, a parabolic and an elliptic one respectively.
	More precisely, the system proposed in \cite{PP2021} takes the following form:
	\begin{alignat*}{2}
		& \partial_t \phi - \div (b(\phi) \nabla (\mu + \Psi'_+(\phi))) = 0 \quad && \hbox{in $Q_T$,} \\
		& - \delta \Delta \mu + \mu = - \gamma \Delta \phi + \Psi'_- \left( \phi - \frac{\delta}{\gamma} \mu \right) \quad && \hbox{in $Q_T$,} \\
		& \partial_{\n} \left( \phi - \frac{\delta}{\gamma} \mu \right) = b(\phi) \partial_{\n} (\mu + \Psi'_+(\phi)) = 0 \quad && \hbox{on $\Sigma_T$,} \\
		& \phi(0) = \phi_0 \quad && \hbox{in $\Omega$.}
	\end{alignat*}
	The convex splitting $\Psi = \Psi_+ + \Psi_-$ allows us to keep the convex and stable part in the parabolic equation for $\phi$ and the concave and unstable part $\Psi_-$ in the (now elliptic) equation for $\mu$.
	In the case of the potential introduced above \eqref{eq:singlewell} we have, 
	for some given $\kappa\in\mathbb{R}$,
    \begin{equation}
    \label{eq:potdecomp}
		\Psi_+(s) = - (1 - s^*) \log(1 - s) - \frac{s^3}{3} + \kappa,
		\quad \Psi_-(s) = - (1 - s^*) \frac{s^2}{2} - (1 - s^*) s,
    \end{equation}
    where $\Psi_+$ is convex if $s^* \in (0,0.7]$. 
    We will tacitly assume this condition on $s^*$ in everything that follows.
	In \cite{PP2021} the authors established the existence of a weak solution for $\delta > 0$. Then, they proved that a suitable sequence converges to a weak solution to the original system as~$\delta \to 0$. This provides a validation of the approximation, which turns to be useful also from the numerical viewpoint. Indeed, in~\cite{BP2022}, the authors derive and analyse a structure-preserving numerical scheme for the relaxed system, showing the validity of the proposed relaxation.

	In this contribution, we extend this relaxation to a more refined phase-field model for tumour growth which accounts for the presence of a nutrient as well as for chemotaxis.
	We recall that many solid tumour growth models based on the Cahn--Hilliard equation are available in the literature (see, e.g., \cite{Fritz2023} and references therein), however, degenerate mobility is rarely employed, despite being relevant from a physical viewpoint, due to the above-mentioned technical difficulties.

	Let us introduce our extended model. The phase variable $\phi$ represents the volume fraction of tumour cells in a tissue, with $\phi=0$ being the healthy phase and $\phi=1$ the cancerous one.
	Then, we model the tumour proliferation through the consumption of a key nutrient $\sigma$, such as oxygen or glucose, whose evolution is ruled by a reaction-diffusion equation.
	Additionally, by including also some cross-diffusion terms weighted by a non-negative parameter $\chi \ge 0$, we introduce chemotactic effects to model the natural movement of tumour cells towards regions with higher nutrient concentrations.
	Summing up, we study the following system:
	\begin{alignat}{2}
		& \partial_t \phid
		- \div (b(\phid) \nabla (\mud + \Psi'_+(\phid)))
		= R_1(\phid, \mud, \sigmad), \label{eq:phi} \\
		& - \delta \Delta \mud + \mud
		= - \gamma \Delta \phid
		+ \Psi'_- \left( \phid - \frac{\delta}{\gamma} \mud \right)
		- \chi \sigmad,
        \label{eq:mu} \\
		& \partial_t \sigmad -
		\Delta \left( \sigmad + \chi \left( 1 - \left(\phid - \frac{\delta}{\gamma} \mud \right) \right) \right)
		= R_2(\phid, \mud, \sigmad),
        \label{eq:sigma}
        \end{alignat}
        in $Q_T$, subject to the boundary and initial conditions
        \begin{alignat}{2}
		& \partial_{\n} \left( \phid - \frac{\delta}{\gamma} \mud \right)
		= b(\phid) \partial_{\n} (\mud + \Psi'_+(\phid))
		= \partial_{\n} \sigmad = 0
		\quad && \hbox{on $\Sigma_T$,} \label{eq:bc} \\
		& \phid(0) = \phi_0,
		\quad \sigmad(0) = \sigma_0
		\quad && \hbox{in $\Omega$.} \label{eq:ic}
	\end{alignat}
	This system is characterised by the following free energy
	\begin{equation}
		\label{eq:freeenergy}
		\begin{split}
		\mathcal{E}_\delta(\phid, \sigmad) & =
		\int_\Omega \left( \Psi_+(\phid)
		+ \frac{\gamma}{2} \abs*{\nabla \left(\phid - \frac{\delta}{\gamma} \mud \right)}^2
		+ \frac{\delta}{2 \gamma} \abs{\mud}^2
		+ \Psi_-\left(\phid - \frac{\delta}{\gamma} \mud\right) \right) \, \de x \\
		& \quad
		+ \int_\Omega \left( \mezzo \abs{\sigmad}^2
		+ \chi \sigmad \left( 1 - \left( \phid - \frac{\delta}{\gamma} \mud \right) \right) \right)
		\, \de x.
		\end{split}
	\end{equation}
	We observe that, even if $\mud$ appears in its expression, for any fixed time $t \in (0,T)$, the free energy actually depends only on $\phid$ and $\sigmad$, since $\mud$ can be determined as the unique solution to the stationary elliptic equation \eqref{eq:mu}.
    Moreover, the system also presents an entropy structure due to the presence of the degenerate mobility. 
    Indeed, one can define an entropy density function $\eta:(0,1) \to \R$ as the solution to the following Cauchy problem
    \begin{equation*}
		\begin{cases}
			b(s) \eta''(s) = 1, \\
			\eta(1/2) = \eta'(1/2) = 0.
		\end{cases}
	\end{equation*}
    Then, the entropy functional is given by 
    \[
        \mathcal{S}(\phid) = \int_\Omega \eta(\phid) \, \de x.
    \]
	Going back to our system, the choice of the reaction terms $R_1(\phid, \mud, \sigmad)$ and $R_2(\phid, \mud, \sigmad)$ is a rather delicate matter, especially in the limiting case $\delta = 0$.
	Indeed, due to very weak regularity that can be established on account of the degenerate mobility, they have to preserve in some way the energy and entropy structure of the system, uniformly in $\delta$.
	For this reason, we adopt the choices made in \cite{FLR2017} for a nonlocal version of the same model with degenerate mobility and singular double-well potential. These choices are based on the original model introduced in~\cite{HZO2012}.
	More precisely, we assume:
	\begin{equation}
		\label{eq:sources}
		\begin{split}
		& R_1(\phid, \mud, \sigmad) = P(\phid) \left( \sigmad + \chi \left( 1 - \left( \phid - \frac{\delta}{\gamma} \mud \right) \right) - (\mud + \Psi'_+(\phid)) \right), \\
		& R_2(\phid, \mud, \sigmad) = - R_1(\phid, \mud, \sigmad),
		\end{split}
	\end{equation}
	where $P(\phid)$ is a suitable non-negative proliferation function.
	The main contribution of these reaction terms is the first one, which models the evolution of the tumour by the consumption of nutrient, while the other two terms are related to the Lyapunov structure of the system.
	In our case, this is fundamental to have an energy inequality independent of $\delta > 0$.
	At the same time, to balance the degenerate mobility and guarantee an entropy structure, the proliferation function $P$ must also be degenerate in $0$ and $1$, namely one can also assume that it has the same expression of $b$, i.~e.~$P(s) = P_0 s (1 - s)^2$ for some $P_0 > 0$.

	It is worth observing that the choice of a degenerate mobility is particularly relevant with these reaction terms, as in the numerical simulations shown in \cite{HZO2012} we can see the formation of branches in the tumour's evolution when chemotactic effects are relevant enough.
	Finally, on account of \eqref{eq:sources}, setting $\delta=0$ in our problem, we formally recover the original tumour growth model proposed in \cite{HZO2012}:
	\begin{alignat}{2}
		& \partial_t \phi
		- \div (b(\phi) \nabla \mu)
		= P(\phi)(\sigma + \chi (1 - \phi) - \mu)
		\quad && \hbox{in $Q_T$,} \label{eq:phi0} \\
		& \mu
		= - \gamma \Delta \phi
		+ \Psi'(\phi)
		- \chi \sigma
		\quad && \hbox{in $Q_T$,} \label{eq:mu0} \\
		& \partial_t \sigma -
		\Delta \left( \sigma + \chi (1 - \phi) \right)
		= - P(\phi)(\sigma + \chi (1 - \phi) - \mu)
		\quad && \hbox{in $Q_T$,} \label{eq:sigma0} \\
		& \partial_{\n} \phi
		= b(\phi) \partial_{\n} \mu
		= \partial_{\n} \sigma = 0
		\quad && \hbox{on $\Sigma_T$,} \label{eq:bc0} \\
		& \phi(0) = \phi_0,
		\quad \sigma(0) = \sigma_0
		\quad && \hbox{in $\Omega$.} \label{eq:ic0}
	\end{alignat}
    We recall that system \eqref{eq:phi0}--\eqref{eq:ic0} was first analysed
    in \cite{FGR2015} in the case of constant mobility and regular potential, without chemotaxis. The singular potential was then considered in \cite{CGH2015}, but in the viscous case. Then, systems like \eqref{eq:phi0}--\eqref{eq:sigma0} have been widely studied in the recent literature, also with respect to its nonlocal variants \cite{F2023_viscous, F2024, GMS2026}, optimal control \cite{CGRS2017, CSS2021_secondorder, DRST2025}, inverse problems \cite{FLS2021, ABCFR2025}, long-time behaviour \cite{CRW2021, GY2020, Ya2025}, and numerical analysis \cite{SLT2025,WLY2025}.
    However, the case of degenerate mobility and a singular single-well potential has remained unexplored so far.

    Our main results include the proof of the existence of a weak solution to \eqref{eq:phi}--\eqref{eq:ic} for $\delta > 0$ (Theorem \ref{thm:weaksols_delta}), and the rigorous limit as $\delta \to 0$ to recover a weak solution to \eqref{eq:phi0}--\eqref{eq:ic0} (Theorem \ref{thm:conv_deltato0}).
   The first result is obtained by extending the approach used in \cite{PP2021}, which is based on suitable approximations of the potential and the mobility, combined with a Galerkin scheme. It is interesting to note that, in the relaxed case $\delta>0$, the chemical potential has a nonlocal spatial dependence on $\phi$. Thus, our second main result can be viewed as a nonlocal-to-local convergence. For a similar issue related to the nonlocal Cahn--Hilliard equation, see, for instance, \cite{DRST2020, DST2021, ES2023, ES2025}.
   We point out that our novel existence result for weak solutions to \eqref{eq:phi0}--\eqref{eq:ic0} is obtained passing to the limit in a non-approximated version of the relaxed problem (cf. \cite{PP2021}), which can be solved numerically in a more efficient way.
    
    The plan of the paper is as follows. In Section \ref{sec:hps_res} we introduce some notation and the basic assumptions, as well as the main results of the paper. Section \ref{apprsols} is devoted to the construction of approximating solutions by regularising the mobility, the potential, and the proliferation function. Sections \ref{relaxprob} and \ref{limitprob} contain the proof of the two main results, respectively. A technical lemma is reported in Appendix.

	\section{Main results}
    \label{sec:hps_res}

    Our first purpose is to introduce some notation that will be used throughout the paper and the key hypotheses on the parameters of the system \eqref{eq:phi}--\eqref{eq:ic}.

    Let $X$ be a real Banach space, and denote by $\norm{\cdot}_X$ its norm, by $X^*$ its dual space, and by $\duality{\cdot,\cdot}_X$ its corresponding duality product.
    If $X$ is a Hilbert space, then we indicate by $(\cdot, \cdot)_X$ its inner product.
    Given a sufficiently smooth bounded domain $\Omega \subset \R^d$, $d = 2,3$, we use the standard notation $\Lx p$ and $\Wx {k,p}$ for the Lebesgue and Sobolev spaces, respectively, for $p \in [1,+\infty]$ and $k \in \N$.
    If $p = 2$, we set $\Hx k = \Wx {k,2}$.
    For any function $f \in \Lx 1$, we denote its mean value as
    \[
        \overline{f} := \frac{1}{\abs{\Omega}} \int_\Omega f \, \de x.
    \]
    Given $T > 0$ and a Banach space $X$, we denote by $\LT p X$ and $\WT {k,p} X$ the corresponding Bochner spaces for $p \in [1,+\infty]$ and $k \in \N$.
    If $p = 2$, we similarly write $\HT k X = \WT {k,2} X$.

    Next, we introduce the Hilbert triplets that will be extensively used in the following. Setting
    \[
        H := \Lx 2, \quad V := \Hx 1, \quad W := \{ u \in \Hx 2 \mid \partial_{\n} u = 0 \hbox{ on $\Omega$} \},
    \]
    and identifying $H$ with its dual space, the following compact and dense embeddings hold:
    \[
        W \subset\subset V \subset\subset H \cong H^* \subset\subset V^* \subset\subset W^*.
    \]
    Additionally, we recall that, by elliptic regularity theory, if $\partial \Omega$ is of class $C^2$, the norm
    \[
        \norm{u}_W^2 := \norm{u}^2_H + \norm{\Delta u}^2_H
    \]
    is equivalent to the $\Hx 2$-norm on $W$.
    Next, we introduce the Riesz isomorphism $\Ncal: V \to V^*$ defined by
    \[
        \duality{\Ncal u, v}_V := \int_\Omega ( \nabla u \cdot \nabla v + u v) \, \de x \quad \hbox{for any $u,v \in V$.}
    \]
    It is well-known that $\Ncal$ is an isomorphism between $V$ and $V^*$, and that its inverse $\Ncal^{-1}$ is the operator which associates to any $f \in V^*$ the weak solution $u \in V$ to the following Neumann--Laplace problem:
    \begin{equation*}
		\begin{cases}
			- \Delta u + u = f \quad & \hbox{in $\Omega$,} \\
			\partial_{\n} u = 0 \quad & \hbox{on $\partial \Omega$.}
		\end{cases}
	\end{equation*}
	In particular, the following identities hold
	\[
		\norm{f}^2_{V^*} = \norm{u}^2_{V},
        \quad \duality{f, \mathcal{N}^{-1} f}_{V} = \norm{f}^2_{V^*},
        \quad \duality{\mathcal{N}u, \mathcal{N}^{-1}u}_{V} = \norm{u}^2_{H}.
	\]
    Moreover, it is well-known that for $u \in W$ we have that $\Ncal u = - \Delta u + u \in H$ and that the restriction of $\Ncal$ to $W$ is an isomorphism from $W$ to $H$.
    Finally, we recall that, by the Spectral Theorem, there exists a sequence of eigenvalues $0 < \lambda_1 \le \lambda_2 \le \dots$, with $\lambda_j \to \infty$ as $j \to \infty$, and a family of eigenfunctions $\{w_j\}_{j \in \N}  \subset W$ such that $\Ncal w_j = \lambda_j w_j$, which forms an orthonormal Schauder basis in $H$ and an orthogonal Schauder basis in $V$.
    In particular, $w_1$ is constant.

    We now introduce the main assumptions that we use throughout our analysis.
	\begin{enumerate}[font = \bfseries, label = A\arabic*., ref = \bf{A\arabic*}]
		\item\label{ass:omega} $\Omega \subset \R^d$, $d = 2,3$, is a bounded domain with $C^2$-boundary and outer normal $\n$, $T > 0$ is given, and $\gamma > 0$ is a fixed parameter.
		\item\label{ass:psi} The single-well potential $\Psi$ can be decomposed in a convex and a concave part as follows:
		\[
			\Psi(s) = \Psi_+(s) + \Psi_-(s) \quad \hbox{ for any $s\in [0,1)$}.
		\]
		We suppose that the singularity is contained in the convex part $\Psi_+$, and we assume that
		\begin{gather*}
			\Psi_- \in C^2([0,1]), \quad \Psi_+ \in C^2([0,1)),
			\quad \Psi''_+(s) > 0 \quad \hbox{for any $0 < s < 1$,}
			\quad \lim_{s \to 1} \Psi_+'(s) = + \infty.
		\end{gather*}
        Moreover, we assume that $\Psi_- \in C^2([0,1])$ can be extended to a smooth function on the full real line such that
		\[
			\Psi_- \in C^2(\R),
			\quad \Psi''_- \in L^\infty(\R),
            \quad \abs{\Psi_-'(s)} \le c_0 \abs{s} + c_0, \quad \hbox{for any $s \in \R$,}
		\]
        and
        \begin{equation}
            \label{ass:coerc}
            \Psi_-(s) \ge c_1 s^2 - c_2, \quad \hbox{for any $s \in \R$,}
        \end{equation}
        for some constants $c_0, c_1, c_2 > 0$.
        One may think of a quadratic extension with a double-well structure. Also, we can assume $\Psi_+$ to be non-negative in $[0,1)$ provided we fix a convenient $\kappa$ (see \eqref{eq:potdecomp}).
        \item\label{ass:psi_small}
		We additionally assume that $\delta$ is sufficiently small, that is,
		\[
			\delta < \delta_0 := \min \left\{ \frac{\gamma^2}{\norm{\Psi''_-}^2_\infty}, \frac12 \right\}.
		\]
		\item\label{ass:b} The degenerate mobility $b \in C^0([0,1])$ satisfies
		\[
			b(0) = b(1) = 0, \quad b(s) > 0 \quad \hbox{for 0 < s < 1,}
		\]
		as well as the compatibility condition
		\[
			b(\cdot) \Psi_+''(\cdot) \in C^0([0,1]).
		\]
        Moreover, we assume that there exists $\eps_0 \in (0,1/2)$ such that $b$ is non-decreasing in $[0,\eps_0]$ and non-increasing in $[1-\eps_0, 1]$. 
        Similarly, we also assume that $\Psi_+''$ is non-increasing in $[0,\eps_0]$ and non-decreasing in $[1-\eps_0,1)$.
		\item\label{ass:P} The proliferation function $P \in C^0([0,1])$ satisfies $P(s) \ge 0$ and the compatibility condition
		\[
			P(\cdot) \Psi_+'(\cdot) \in C^0([0,1]).
		\]
        Moreover, we assume that
        \[
            \sqrt{P(s)} \le c_3 b(s) \quad \hbox{for any $0 \le s \le \eps_0$ and $1-\eps_0\le s \le 1$,}
        \]
		for some constant $c_3 > 0$, where $\eps_0 \in (0,1/2)$ (see \ref{ass:b}).
		\item\label{ass:chi} Given the chemotaxis parameter $\chi \ge 0$, the constants $c_1$ and $c_2$ in \eqref{ass:coerc} are chosen such that $ c_1 > \chi^2$ (see \ref{ass:psi}).
		\item\label{ass:iniz} $\phi_0 \in \Lx 2$ with $0 \le \phi_0 < 1$ is such that $\Psi(\phi_0) \in \Lx 1$, and $\sigma_0 \in \Lx2$.
        \item\label{ass:entropy} We define the entropy density function $\eta:(0,1) \to \R$ as the solution to the following Cauchy problem
        \begin{equation*}
		\begin{cases}
			b(s) \eta''(s) = 1, \\
			\eta(1/2) = \eta'(1/2) = 0.
		\end{cases}
	    \end{equation*}
        Then, we assume that $\eta(\phi_0) \in \Lx 1$.
	\end{enumerate}

    \begin{remark}
        Regarding the smallness condition on the relaxation parameter $\delta$ given in hypothesis \ref{ass:psi_small}, we observe that the bound by 1/2 appears mainly for technical reasons. 
        However, it is not restrictive in practice.
        Indeed, the actual smallness condition is given by the first bound, since $\gamma$ is the interface parameter in the Cahn--Hilliard equation, which is usually taken very small in applications.
    \end{remark}


    \begin{remark}
    \label{rmk:Petaprime}
        By defining the entropy density function $\eta$ as in assumption
        \ref{ass:entropy}, we can see that hypotheses \ref{ass:b} and \ref{ass:P} imply the existence of two constants $c_4, c_5 > 0$ such that
        \begin{equation}
		  \label{ass:Peta}
		      \abs{\sqrt{P(s)} \eta'(s)} \le c_4 \abs{s} + c_5 \quad \hbox{for any $s \in [0,1]$.}
	    \end{equation}
       Indeed, the estimate is trivial if $s \in (\eps_0, 1 - \eps_0)$, due to the boundedness of $P$ and the non-degeneracy of $b$.
       For $s \in [1- \eps_0, 1)$, we proceed as follows:
       \begin{align*}
            \abs{\sqrt{P(s)} \eta'(s)}
            & \le \abs*{\sqrt{P(s)} \int_{1/2}^s \frac{1}{b(r)} \, \de r}
            \le \abs*{\sqrt{P(s)} \left(\int_{1/2}^{1-\eps_0} \frac{1}{b(r)} \, \de r + \int_{1-\eps_0}^{s} \frac{1}{b(r)} \, \de r\right)} \\
            & \le \norm{\sqrt{P}}_\infty C + c_3 b(s) \frac{s - (1-\eps_0)}{b(s)}
            \le c_3 \abs{s} + \bar{C}_1,
       \end{align*}
       where we used the fact that $b$ is non-increasing in $[1-\eps_0, 1]$ and \ref{ass:P}.
       For $s \in (0,\eps_0]$ instead, by similar arguments, we infer that
       \begin{align*}
            \abs{\sqrt{P(s)} \eta'(s)}
            & \le \abs*{\sqrt{P(s)} \int_s^{1/2} \frac{1}{b(r)} \, \de r}
            \le \abs*{\sqrt{P(s)} \left(\int_{s}^{\eps_0} \frac{1}{b(r)} \, \de r + \int_{\eps_0}^{1/2} \frac{1}{b(r)} \, \de r\right)} \\
            & \le c_3 b(s) \frac{\eps_0 - s}{b(s)} + \norm{\sqrt{P}}_\infty C
            \le c_3 \abs{s} + \bar{C}_2.
       \end{align*}
       This proves \eqref{ass:Peta} with $c_4 = c_3$ and $c_5 = \max\{\bar{C}_1, \bar{C}_2\}$.
    \end{remark}

    \noindent
    In what follows, we will extensively denote by $C > 0$ a positive constant depending only on the fixed quantities introduced in the hypotheses above.
    Such constants may change from line to line, depending on the context.
    We may use subscripts to highlight some particular dependencies.



    We now state the main results of the paper.
    Our first result is existence of a weak solution to the relaxed problem \eqref{eq:phi}--\eqref{eq:ic}.

    \begin{theorem}
        \label{thm:weaksols_delta}
        Let \ref{ass:omega}--\ref{ass:entropy} hold.
        Then, problem \eqref{eq:phi}--\eqref{eq:ic} admits a weak solution $(\phid, \mud, \sigmad)$, such that
        \begin{align*}
            & \phid \in \HT 1 {V^*} \cap \LT 2 V, \\
            & \mud \in \LT 2 V, \\
            & \phid - \dg \mud \in \HT 1 {V^*} \cap \LT \infty V \cap \LT 2 W, \\
            & \sigmad \in \HT 1 {V^*} \cap \LT \infty H \cap \LT 2 V, \\
            & 0 \le \phid \le 1 \quad \hbox{a.e.~in $Q_T$,}
        \end{align*}
        and that, for any $v \in V$ and almost everywhere in $(0,T)$,
        \begin{align}
            & \duality{\partial_t \phid, v}_V
            + (b(\phid) (\nabla \mud + \Psi_+''(\phid) \nabla \phid), \nabla v)_H \notag \\
            & \quad = \left( P(\phid) \left( \sigmad + \chi \left( 1 - \left( \phid - \dg \mud \right) \right) - (\mud + \Psi_+'(\phid)) \right) , v \right)_{\!\! H}
            \label{eq:varform:phi}\\
            & \delta (\nabla \mud, \nabla v)_H
            + (\mud, v)_H \notag \\
            & \quad = \gamma (\nabla \phid, \nabla v)_H
            + \left( \Psi'_- \left( \phid - \frac{\delta}{\gamma} \mud \right), v \right)_{\!\! H}
            - \chi (\sigmad, v)_H
            \label{eq:varform:mu} \\
            & \duality{\partial_t \sigmad, v}_V
            + \left( \nabla \left( \sigmad + \chi \left( 1 - \left(\phid - \frac{\delta}{\gamma} \mud \right) \right) \right), v \right)_{\!\! H} \notag \\
            & \quad = - \left( P(\phid) \left( \sigmad + \chi \left( 1 - \left( \phid - \frac{\delta}{\gamma} \mud \right) \right) - (\mud + \Psi_+'(\phid)) \right) , v \right)_{\!\! H}
            \label{eq:varform:sigma}
        \end{align}
       with initial conditions $\phid(0) = \phi_0$ and $\sigmad(0) = \sigma_0$.
    \end{theorem}

    \begin{remark}
    In general, due to the structure of the reaction terms, we cannot ensure that $\sigma_\delta \in  [0,1]$. However, testing both equations with 1 and adding them up, one can see that the total mass of $\phi_\delta + \sigma_\delta$ is conserved over time.
    \end{remark}

    \begin{remark}
    \label{rmk:initial_data}
        A comment on the initial datum $\mud(0)$ for the relaxed chemical potential is in order. 
        Indeed, such initial datum is not prescribed, but, by the regularities given above, one can easily deduce that $\phid - \dg \mud \in \CT 0 V$, meaning that $\phi_0 - \dg \mud(0) \in V$ should be well-defined. 
        Indeed, given $\phi_0 \in H$ and $\sigma_0 \in H$, $\phi_0 - \dg \mud(0) \in V$ can be determined as the unique weak solution to the following boundary value problem, which can be deduced from \eqref{eq:mu} at time $t=0$ by multiplying by $\delta/\gamma$ and by adding and subtracting $\phi_0$, namely
        \begin{align*}
            & - \delta \Delta \left(\phi_0 - \dg \mud(0) \right)
            + \left( \phi_0 - \dg \mud(0) \right)
            = \phi_0
            - \dg \Psi'_- \left(\phi_0 - \dg \mud(0)\right) + \dg \chi \sigma_0, \quad\hbox{ in }\Omega, \\
            & \partial_{\n} \left(\phi_0 - \dg \mud(0) \right) = 0, \quad \hbox{ on }\partial\Omega.
        \end{align*}
        Then, by the regularity hypotheses on $\Psi_-'$, one can uniquely find $\phi_0 - \dg \mud(0) \in V$, which then gives that $\mud(0) \in H$, since $\phi_0 \in H$.
        This also crucially implies that the initial energy
        \begin{align*}
    		\mathcal{E}_\delta(\phi_0, \sigma_0) & =
    		\int_\Omega \left( \Psi_+(\phi_0)
    		+ \frac{\gamma}{2} \abs*{\nabla \left(\phi_0 - \frac{\delta}{\gamma} \mud(0) \right)}^2
    		+ \frac{\delta}{2 \gamma} \abs{\mud(0)}^2
    		+ \Psi_-\left(\phi_0 - \frac{\delta}{\gamma} \mud(0) \right) \right) \, \de x \\
    		& \quad
    		+ \int_\Omega \left( \mezzo \abs{\sigma_0}^2
    		+ \chi \sigmad \left( 1 - \left( \phi_0 - \frac{\delta}{\gamma} \mud(0) \right) \right) \right)
    		\, \de x
	   \end{align*}
       is bounded under our hypotheses. It is worth pointing out that, in the relaxed problem, it is enough to assume that $\phi_0 \in H$. This is related to relaxation of the chemical potential which transforms the original Cahn--Hilliard equation into a second-order system.
       
    \end{remark}

    \begin{remark}
    \label{strict}
        If we additionally assume that the mobility $b$ degenerates fast enough at $1$ so that $\lim_{s \to 1^-} \eta(s) = + \infty$, we can also show that 
        \[
            \phid < 1 \quad \hbox{a.e.~in $Q_T$,}
        \]
        exactly as done in the proof of \cite[Theorem 5]{PP2021}, due to the boundedness of the entropy.
        In particular, we observe that this condition is satisfied if we take the standard degenerate mobility $b(s) = s(1-s)^2$, since in this case $\eta(s) \sim - \log(1 - s)$ as $s \to 1^-$.
    \end{remark}

    Our second result shows that the relaxed model \eqref{eq:phi}--\eqref{eq:ic} is a consistent approximation of model \eqref{eq:phi0}--\eqref{eq:ic0}.
    Namely, we prove that, as $\delta \to 0$, a weak solution to \eqref{eq:phi}--\eqref{eq:ic} given by Theorem \ref{thm:weaksols_delta} converges, up to a subsequence, to a suitable notion of solution to problem \eqref{eq:phi0}--\eqref{eq:ic0}.

    Let us introduce a convenient notion of weak solution to the original problem \eqref{eq:phi0}--\eqref{eq:ic0}, following the approach pioneered in \cite{EG1996} and later used also in \cite{AACG2017, PP2021}.
    The main issue is that, due to the degeneracy of the mobility, one is not able to control $\nabla \mu$ in $\LT 2 H$ from the energy estimate.
    Actually, one is only able to bound the flux
    \[
        \vec{J} := b(\phi) \nabla \mu = b(\phi) (\nabla (- \gamma \Delta \phi)
        + \Psi''(\phi) \nabla \phi
        - \chi \nabla \sigma )
    \]
    uniformly in $\LT 2 H$.
    Thus, the system has to be formally rewritten only in terms of the two main variables $\phi$ and $\sigma$, namely:
    \begin{alignat}{2}
		& \partial_t \phi
		- \div (b(\phi) (\nabla (- \gamma \Delta \phi)
        + \Psi''(\phi) \nabla \phi
        - \chi \nabla \sigma )) \notag \\
		& \quad = P(\phi)(\sigma + \chi (1 - \phi) + \gamma \Delta \phi - \Psi'(\phi) + \chi \sigma)
		\quad && \hbox{in $Q_T$,} \label{eq:phi00} \\
		& \partial_t \sigma -
		\Delta \left( \sigma + \chi (1 - \phi) \right)
		= - P(\phi)(\sigma + \chi (1 - \phi) + \gamma \Delta \phi - \Psi'(\phi) + \chi \sigma)
		\quad && \hbox{in $Q_T$,} \label{eq:sigma00} \\
		& \partial_{\n} \phi
		= b(\phi) \partial_{\n} (- \gamma \Delta \phi
		+ \Psi'(\phi)
		- \chi \sigma)
		= \partial_{\n} \sigma = 0
		\quad && \hbox{on $\Sigma_T$,} \label{eq:bc00} \\
		& \phi(0) = \phi_0,
		\quad \sigma(0) = \sigma_0
		\quad && \hbox{in $\Omega$.} \label{eq:ic00}
	\end{alignat}
    Then, by assuming the additional hypotheses
    \begin{enumerate}[font = \bfseries, label = A\arabic*., ref = \bf{A\arabic*}]
    \setcounter{enumi}{8}
        \item\label{ass:b_c1} $b \in C^1([0,1])$,
        \item\label{ass:iniz2} $\phi_0 \in V$,
    \end{enumerate}
    one can introduce the following notion of weak solution to problem \eqref{eq:phi00}--\eqref{eq:ic00}.

    \begin{definition}
        \label{def:weaksol_limit}
        Given $\phi_0$ and $\sigma_0$ satisfying \ref{ass:iniz} and \ref{ass:iniz2}, we say that the triple $(\phi, \vec{J}, \sigma)$ such that
        \begin{align*}
            & \phi \in \HT 1 {V^*} \cap \CT 0 V \cap \LT 2 {\Hx 2}, \\
            & \vec{J} \in \LT 2 H, \\
            & \sigma \in \HT 1 {V^*} \cap \CT 0 H \cap \LT 2 V,
        \end{align*}
        is a weak solution to \eqref{eq:phi00}--\eqref{eq:ic00} if the following variational formulation is satisfied for any $v \in V$, any $\vec{\xi} \in H^1(\Omega; \R^d) \cap L^\infty(\Omega; \R^d)$:
        \begin{align}
            & \duality{\partial_t \phi, v}_V
            + (\vec{J}, \nabla v)_H
            = (P(\phi)(\sigma + \chi (1 - \phi) + \gamma \Delta \phi - \Psi'(\phi) + \chi \sigma), v)_H, \label{eq:varform:phi00} \\
            & (\vec{J}, \vec{\xi})_H
            = - (\gamma \Delta \phi \, b'(\phi) \nabla \phi, \vec{\xi})_H
            - (\gamma b(\phi) \Delta \phi, \div \vec{\xi})_H \notag \\
            & \qquad + (b(\phi) \Psi''(\phi) \nabla \phi, \vec{\xi})_H
            - \chi (b(\phi) \nabla \sigma, \vec{\xi})_H,
            \label{eq:varform:J} \\
            & \duality{\partial_t \sigma, v}_V
            + (\nabla (\sigma + \chi (1 - \phi)), \nabla v)_H \notag \\
            & \qquad = - (P(\phi)(\sigma + \chi (1 - \phi) + \gamma \Delta \phi - \Psi'(\phi) + \chi \sigma), v)_H, \label{eq:varform:sigma00}
        \end{align}
    almost everywhere in $(0,T)$.
    \end{definition}

    \noindent
    We are now in a position to state our second main result.

    \begin{theorem}
    \label{thm:conv_deltato0}
       Let \ref{ass:omega}--\ref{ass:entropy} and \ref{ass:b_c1}--\ref{ass:iniz2} hold.
        For any $\delta > 0$, let $(\phid, \mud, \sigmad)$ be a weak solution to problem \eqref{eq:phi}--\eqref{eq:ic} in the sense of Theorem \emph{\ref{thm:weaksols_delta}}.
        Then, as $\delta \to 0$, one can extract a non-relabelled subsequence of $\{ (\phid, \mud, \sigmad) \}_{\delta > 0}$ such that
        \begin{align*}
            & \vec{J}_\delta = b(\phid) \nabla \mud + b(\phid)\Psi_+''(\phid) \nabla \phid \weak \vec{J} \quad \hbox{weakly in $\LT 2 H$,} \\
            & \mud \weak - \gamma \Delta \phi + \Psi_-'(\phi) - \chi \sigma \quad \hbox{weakly in $\LT 2 H$,} \\
            & \phid - \dg \mud \weak \phi \quad \hbox{weakly in $\LT 2 {\Hx2}$,} \\
            & \phid - \dg \mud \to \phi
            \quad \text{and} \quad
            \phid \to \phi
            \quad \hbox{strongly in $\LT 2 V$,} \\
            & \partial_t \phid \weak \partial_t \phi \quad \hbox{weakly in $\LT 2 {V^*}$,} \\
            & \sigmad \weak \sigma \quad \hbox{weakly in $\HT 1 {V^*} \cap \LT 2 V$,} \\
            & 0 \le \phi \le 1 \quad \hbox{a.e.~in $Q_T$,}
        \end{align*}
        and $(\phi, \vec{J}, \sigma)$ is a weak solution to problem \eqref{eq:phi00}--\eqref{eq:ic00} in the sense of Definition \emph{\ref{def:weaksol_limit}}.
    \end{theorem}

    \begin{remark}
        The main ingredient for the proof of Theorem \ref{thm:conv_deltato0} is the validity of some uniform bounds on the solutions, provided by the energy and entropy structure of the relaxed system \eqref{eq:phi}--\eqref{eq:ic}. 
        Indeed, from the energy estimate we can deduce that, for almost any $t \in (0,T)$, 
        \begin{equation*}
            \begin{split}
            & \beta \int_\Omega \abs*{\phid(t) - \dg \mud(t)}^2 \de x
            + \mezzo \dg \int_\Omega \abs{\mud(t)}^2 \, \de x \\
            & \qquad
            + \frac{\gamma}{2} \int_\Omega \abs*{\nabla \left(\phid(t) - \dg \mud(t) \right)}^2 \,\de x
            + \alpha \int_\Omega \abs{\sigmad(t)}^2 \, \de x \\
            & \qquad
            + \int_0^t \int_\Omega \abs*{\nabla \left(\sigmad + \chi \left( 1 - \left(\phid - \dg \mud \right) \right) \right)}^2 \, \de x \, \de s
            \\
            & \qquad
            + \int_0^t \int_\Omega \left[ \sqrt{P(\phid)} \left( \sigmad + \chi \left( 1 - \left( \phid - \frac{\delta}{\gamma} \mud \right) \right) - (\mud + \Psi_+'(\phid)) \right) \right]^2 \, \de x \, \de s \\
            & \quad \le
            \mathcal{E}_{\delta}(\phi_0, \sigma_0) + C \abs{\Omega},
            \end{split}
        \end{equation*}
        for some $\alpha, \beta > 0$ and $C > 0$ depending only on the parameters of the system, but independent of $T$ and of $\delta$. In this case, assumption \ref{ass:iniz2} ensures that the initial energy $\mathcal{E}_0(\phi_0, \sigma_0)$ for $\delta=0$ is bounded, thus the right-hand side can be bounded uniformly in $\delta$ (see also Remark \ref{rmk:initial_data}).
        Additionally, from the entropy estimate we get that 
        \begin{align*}
            & \frac{\gamma}{2} \int_0^T \int_\Omega \abs*{ \Delta \left( \phid - \dg \mud \right)}^2 \, \de x \, \de t 
            + \dg \int_0^T \int_\Omega \abs{\nabla \mud}^2 \, \de x \, \de t \\
            & \quad \le \mathcal{S}(\phi_0) + C_1 \mathcal{E}_{\delta}(\phi_0, \sigma_0) + C_2,
        \end{align*}
        for some $C_1, C_2 > 0$ depending only on the parameters of the system, but independent of $\delta$.
        These estimates are rigorously derived in the sections below, starting from the corresponding ones for an approximated version of the relaxed system.
    \end{remark}

    \begin{remark}
        Also in this case, if we additionally assume that the mobility $b$ degenerates fast enough at $1$ so that $\lim_{s \to 1^-} \eta(s) = + \infty$, we can also show that 
        \[
            \phi < 1 \quad \hbox{a.e.~in $Q_T$,}
        \]
        using the argument recalled in Remark~\ref{strict}.
    \end{remark}

    \section{Approximating solutions}
    \label{apprsols}

    Here we construct approximating solutions to the relaxed problem \eqref{eq:phi}--\eqref{eq:ic}. More precisely,
    following \cite{EG1996}, we consider a regularised problem with non-degenerate mobility and regular potential.
    Then, we prove a corresponding existence result. The proof of Theorem  \ref{thm:weaksols_delta} will be obtained  in the next section by letting the regularisation parameter vanish.

    Let us fix a small positive parameter $0 < \eps \ll 1$ and define the regularised mobility as follows:
    \begin{equation}
        \label{eq:reg_mob}
        b_\eps(s) =
        \begin{cases}
            b(\eps) \quad & \hbox{for $s \le \eps$,} \\
            b(s) \quad & \hbox{for $\eps \le s \le 1 - \eps$,} \\
            b(1 - \eps) \quad & \hbox{for $s \ge 1 - \eps$.}
        \end{cases}
    \end{equation}
    Then, we can find constants $b_1^\eps, b_2^\eps > 0$ such that
    \begin{equation}
        \label{ass:b_eps}
        b_\eps \in C(\R; \R^+), \quad b_1^\eps \le b_\eps(s) \le b_2^\eps \quad \hbox{for any $s \in \R$.}
    \end{equation}
    To define a regular potential, instead, we smooth out the singularity at $1$, which only occurs in $\Psi_+$, but we also truncate near $0$ to preserve the validity of \ref{ass:b}.
    Indeed, by using a second-order Taylor expansion, we define:
    \begin{equation}
        \label{eq:reg_pot}
        \Psipe(s) =
        \begin{cases}
            \Psi_+(\eps) + \Psi_+'(\eps) (s - \eps) + \mezzo \Psi_+''(\eps) (s - \eps)^2 & \hbox{for $s \le \eps$,} \\
            \Psi_+(s) \quad & \hbox{for $\eps \le s \le 1 - \eps$,} \\
            \Psi_+(1 - \eps) + \Psi_+'(1 - \eps) (s - 1 + \eps) + \mezzo \Psi_+''(1 - \eps) (s - 1 + \eps)^2 & \hbox{for $s \ge 1 - \eps$,}
        \end{cases}
    \end{equation}
    which implies that
    \begin{equation}
        \label{eq:reg_pot_second}
        \Psipe''(s) =
        \begin{cases}
            \Psi_+''(\eps) \quad & \hbox{if $s \le \eps$,} \\
            \Psi_+''(s) \quad & \hbox{if $\eps \le s \le 1 - \eps$,} \\
            \Psi_+''(1 - \eps) \quad & \hbox{if $s \ge 1 - \eps$.}
        \end{cases}
    \end{equation}
    Then, by \ref{ass:psi}, the regularised potential $\Psipe$ satisfies $\Psipe \in C^2(\R)$ and
    \begin{equation}
        \label{ass:Psi_eps}
        \Psipe(s) \ge  - \kappa,
        \quad \abs{\Psipe'(s)} \le c_\eps (1 + \abs{s}),
        \quad \Psipe''(s) \ge \bar{c}_\eps,
        \quad \hbox{for any $s \in \R$,}
    \end{equation}
    where $\kappa, c_\eps, \bar{c}_\eps > 0$ are positive constants, with $\kappa$ independent of $\eps$.
    Additionally, since $\Psi_+$ is convex and $\Psipe$ is defined through its quadratic approximation at $\eps$ and $1 - \eps$, when $\eps$ is small enough, we also have
    \[
        \Psipe(s) \le \Psi_+(s) \quad \hbox{for any $s \in [0,1]$.}
    \]
    Indeed, this is trivial for $s \in (\eps, 1-\eps)$. Then, by performing a Taylor expansion with Lagrange remainder centred in $\eps$, we have that, for any $s \in [0,\eps]$
    \[
        \Psi_+(s) = \Psi_+(\eps) + \Psi_+'(\eps) (s - \eps) + \mezzo \Psi_+''(\theta) (s - \eps)^2,
    \]
    for some $\theta \in (0,\eps)$. 
    Hence, if $\eps < \eps_0$, \ref{ass:b} entails that $\Psi_+''(\theta) \ge \Psi_+''(\eps)$, therefore it follows that 
    \[
        \Psi_+(s) \ge \Psi_+(\eps) + \Psi_+'(\eps) (s - \eps) + \mezzo \Psi_+''(1-\eps) (s - \eps)^2 = \Psipe(s),
    \]
    for any $s \in [0,\eps]$.
    By reasoning in a similar way, it follows also for $s \in [1-\eps, 1)$.
    At the same time, we also extend the concave part $\Psi_-$ to all $\R$ (see \ref{ass:psi}) and we define $\Psi_\eps = \Psipe + \Psi_-$ on the whole real line.
    Then, by our particular choices, the compatibility condition in \ref{ass:b} is also satisfied at the approximate level for any $\eps$ small enough, namely
    \begin{equation}
        \label{ass:bpsi_eps}
        b_\eps(\cdot) \Psi''_\eps(\cdot) \in C^0(\R),
        \quad \text{with} \quad
        \norm{b_\eps \Psi''_\eps}_{L^\infty(\R)} \le \norm{b \Psi''}_{L^\infty([0,1])}.
    \end{equation}
    Moreover, it is easy to see from \eqref{eq:reg_pot} that
    \begin{equation}
        \label{ass:iniz_eps_psi}
        \norm{\Psi_\eps(\phi_0)}_{\Lx 1} \le \norm{\Psi(\phi_0)}_{\Lx 1} + C \left( \norm{\phi_0}^2_H + 1 \right).
    \end{equation}
    Thus, thanks to \ref{ass:iniz}, we have $\Psi_\eps(\phi_0) \in \Lx 1$.
    For compatibility reasons with the approximate mobility $b_\eps$, we also need to regularise the proliferation function $P_\eps$ as follows:
    \begin{equation}
        \label{eq:reg_prol}
        P_\eps(s) =
        \begin{cases}
            P(\eps) \quad & \hbox{for $s \le \eps$,} \\
            P(s) \quad & \hbox{for $\eps \le s \le 1 - \eps$,} \\
            P(1 - \eps) \quad & \hbox{for $s \ge 1 - \eps$.}
        \end{cases}
    \end{equation}
    A similar approach was also used in \cite[Section 4]{FLR2017}.
    Then, the compatibility condition in \ref{ass:P} is still satisfied at the approximate level, namely $P_\eps(\cdot) \Psi'_\eps(\cdot) \in C^0(\R)$.
    Finally, we introduce the approximate entropy density function $\eta_\eps:\R \to \R$ as the solution to the following Cauchy problem
    \begin{equation}
        \label{eq:reg_entr}
		\begin{cases}
			b_\eps(s) \eta_\eps''(s) = 1, \\
			\eta_\eps(1/2) = \eta_\eps'(1/2) = 0.
		\end{cases}
	\end{equation}
    As $b_\eps$ satisfies \eqref{ass:b_eps}, $\eta_\eps \in C^2(\R)$ is well-defined by \eqref{eq:reg_entr}
    Moreover, using \eqref{eq:reg_mob}, one can see that $\eta_\eps$ has the form:
    \begin{equation}
    \label{eq:reg_entr_expl}
        \eta_\eps(s) =
        \begin{cases}
            \eta(\eps) + \eta'(\eps) (s - \eps) + \mezzo \eta''(\eps) (s - \eps)^2 & \hbox{for $s \le \eps$,} \\
            \eta(s) \quad & \hbox{for $\eps \le s \le 1 - \eps$,} \\
            \eta(1 - \eps) + \eta'(1 - \eps) (s - 1 + \eps) + \mezzo \eta''(1 - \eps) (s - 1 + \eps)^2 & \hbox{for $s \ge 1 - \eps$.}
        \end{cases}
    \end{equation}
    Then, since by \ref{ass:b}, $\eta'' = 1/b$ is non-increasing in $(0,\eps_0]$ and non-decreasing in $[1 - \eps_0, 1)$, when $\eps < \eps_0$, by a simple Taylor expansion argument, as done above with $\Psi_+$, one can see that
    \begin{equation}
        \label{ass:iniz_eps_eta}
        \eta_\eps(s) \le \eta(s) \quad \hbox{for any $s \in (0,1)$},
    \end{equation}
    so that $\norm{\eta_\eps(\phi_0)}_{\Lx 1} \le \norm{\eta(\phi_0)}_{\Lx 1}$.
    This implies that $\eta_\eps(\phi_0) \in \Lx 1$, if \ref{ass:entropy} holds.
    To conclude, we also mention that, for $\eps < \eps_0$, the argument in Remark \ref{rmk:Petaprime} can be repeated to deduce that
    \begin{equation}
        \label{ass:Peta_eps}
        \abs{\sqrt{P_\eps(s)} \eta_\eps'(s)} \le c_4 \abs{s} + c_5
        \quad \hbox{for any $s \in \R$,}
    \end{equation}
    where $c_4, c_5$ are the same constants appearing in \eqref{ass:Peta} and they are independent of $\eps$.

    We can now introduce the following approximating problem:
    \begin{alignat}{2}
		& \partial_t \phide
		- \div (b_\eps(\phide) \nabla (\mude + \Psipe'(\phide))) \notag \\
		& \quad = P_\eps(\phide) \left( \sigmade + \chi \left( 1 - \left( \phide - \frac{\delta}{\gamma} \mude \right) \right) - (\mude + \Psipe'(\phide)) \right)
        \quad && \hbox{in $Q_T$,} \label{eq:phieps} \\
		& - \delta \Delta \mude + \mude
		= - \gamma \Delta \phide
		+ \Psi'_- \left( \phide - \frac{\delta}{\gamma} \mude \right)
		- \chi \sigmade
        \quad && \hbox{in $Q_T$,} \label{eq:mueps} \\
		& \partial_t \sigmade -
		\Delta \left( \sigmade + \chi \left( 1 - \left(\phide - \frac{\delta}{\gamma} \mude \right) \right) \right) \notag \\
		& \quad = - P_\eps(\phide) \left( \sigmade + \chi \left( 1 - \left( \phide - \frac{\delta}{\gamma} \mude \right) \right) - (\mude + \Psipe'(\phide)) \right)
        \quad && \hbox{in $Q_T$,} \label{eq:sigmaeps} \\
		& \partial_{\n} \left( \phide - \frac{\delta}{\gamma} \mude \right)
		= b_\eps(\phid) \partial_{\n} (\mude + \Psipe'(\phide))
		= \partial_{\n} \sigmade = 0
		\quad && \hbox{on $\Sigma_T$,} \label{eq:bceps} \\
		& \phide(0) = \phi_0,
		\quad \sigmade(0) = \sigma_0
		\quad && \hbox{in $\Omega$.} \label{eq:iceps}
	\end{alignat}
    The existence of a weak solution to \eqref{eq:phieps}--\eqref{eq:iceps} is given by

    \begin{proposition}
        \label{prop:weaksols_eps}
        Let \ref{ass:omega}--\ref{ass:entropy} be satisfied and let $b_\eps$, $\Psipe$, $P_\eps$ and $\eta_\eps$ be defined  by \eqref{eq:reg_mob}, \eqref{eq:reg_pot}, \eqref{eq:reg_prol}, \eqref{eq:reg_entr}, respectively.
        Then, problem \eqref{eq:phieps}--\eqref{eq:iceps} admits a weak solution, that is, a triplet $(\phide, \mude, \sigmade)$ such that
        \begin{align*}
            & \phide \in \HT 1 {V^*} \cap \LT 2 V, \\
            & \mude \in \LT 2 V, \\
            & \phide - \dg \mude \in \HT 1 {V^*} \cap \LT \infty V \cap \LT 2 W, \\
            & \sigmade \in \HT 1 {V^*} \cap \LT \infty H \cap \LT 2 V,
        \end{align*}
        which satisfies, for any $v \in V$ and almost everywhere in $(0,T)$, the following identities
        \begin{align}
            & \duality{\partial_t \phide, v}_V
            + (b_\eps(\phide) \nabla (\mude + \Psipe'(\phide)), \nabla v)_H \notag \\
            & \quad = \left( P_\eps(\phide) \left( \sigmade + \chi \left( 1 - \left( \phide - \frac{\delta}{\gamma} \mude \right) \right) - (\mude + \Psipe'(\phide)) \right) , v \right)_{\!\! H}
            \label{eq:varform:phi_eps}\\
            & \delta (\nabla \mude, \nabla v)_H
            + (\mude, v)_H \notag \\
            & \quad = \gamma (\nabla \phide, \nabla v)_H
            + \left( \Psi'_- \left( \phide - \frac{\delta}{\gamma} \mude \right), v \right)_{\!\! H}
            - \chi (\sigmade, v)_H
            \label{eq:varform:mu_eps} \\
            & \duality{\partial_t \sigmade, v}_V
            + \left( \nabla \left( \sigmade + \chi \left( 1 - \left(\phide - \frac{\delta}{\gamma} \mude \right) \right) \right), v \right)_{\!\! H} \notag \\
            & \quad = - \left( P_\eps(\phide) \left( \sigmade + \chi \left( 1 - \left( \phide - \frac{\delta}{\gamma} \mude \right) \right) - (\mude + \Psipe'(\phide)) \right) , v \right)_{\!\! H}
            \label{eq:varform:sigma_eps}
        \end{align}
        with initial conditions $\phide(0) = \phi_0$ and $\sigmade(0) = \sigma_0$ in $H$.
    \end{proposition}

    \begin{proof}
        We use an approach based on a Faedo--Galerkin discretisation, as similarly done in \cite{PP2021} for the relaxed Cahn--Hilliard model (see also \cite{AACG2017,ABCG2025} and \cite{FLR2017} for the nonlocal version).
        Hence, we consider the eigenfunctions $\{w_j\}_{j \in \N} \subset W$ of the Neumann--Laplace operator $\Ncal$ introduced in Section \ref{sec:hps_res}, and we define the spaces $W_n := \Span \{ w_1, \dots, w_n \}$ for any $n \in \N$.
        It is well-known that $\bigcup_{n \in \N} W_n$ is dense in $W$, $V$ and $H$.
        We also call $\Pi^n: H \to W_n$ the $\Lx 2$-projection operator onto $W_n$ for any $n \in \N$.
        Since the spaces $W_n$ are generated by the eigenfunctions of $\Ncal$, the $L^2(\Omega)$-projection operator $\Pi^n$ onto $W_n$ has the important property that it coincides with the $H^1(\Omega)$-projection operator onto $W_n$, as one can easily verify by direct calculation.
        Then, for any $n \in \N$, we look for a  $\phiden, \muden, \sigmaden \in \CT 1 {W_n}$ of the form
        \begin{equation}
            \label{def:discr_sols}
            \phiden(t) := \sum_{j=1}^n a_j^n(t) w_j, \quad
            \muden(t) := \sum_{j=1}^n c_j^n(t) w_j, \quad
            \sigmaden(t) := \sum_{j=1}^n d_j^n(t) w_j,
        \end{equation}
        where $a_j^n, c_j^n, d_j^n: [0,T] \to \R$ for any $j = 1, \dots, n$. Such functions are determined by solving the following discretised system for any test function $v \in W_n$ and for almost any $t \in [0,T]$:
        \begin{align}
            & ({\phiden}'(t), v)_H
            + (b_\eps(\phiden(t)) \nabla (\muden(t) + \Pi^n(\Psipe'(\phiden(t)))), \nabla v)_H \notag \\
            & \quad = (P_\eps(\phiden(t)) R^n_{\delta, \eps} (\phiden(t), \muden(t), \sigmaden(t)), v)_H
            \label{eq:varform:phi_eps_n}\\
            & \delta (\nabla \muden(t), \nabla v)_H
            + (\muden(t), v)_H \notag \\
            & \quad = \gamma (\nabla \phiden(t), \nabla v)_H
            + \left( \Psi'_- \left( \phiden(t) - \frac{\delta}{\gamma} \muden(t) \right), v \right)_{\!\! H}
            - \chi (\sigmaden(t), v)_H
            \label{eq:varform:mu_eps_n} \\
            & ({\sigmaden}'(t), v)_H
            + \left( \nabla \left( \sigmaden(t) + \chi \left( 1 - \left(\phiden(t) - \frac{\delta}{\gamma} \muden(t) \right) \right) \right), v \right)_{\!\! H} \notag \\
            & \quad = - (P_\eps(\phiden(t)) R^n_{\delta, \eps} (\phiden(t), \muden(t), \sigmaden(t)), v)_H
            \label{eq:varform:sigma_eps_n} \\
            & \phiden(0) = \Pi^n(\phi_0), \quad \sigmaden(0) = \Pi^n(\sigma_0)
            \label{eq:varform:ic_eps_n}
        \end{align}
        where
        \begin{align*}
            & R^n_{\delta, \eps}(\phiden(t), \muden(t), \sigmaden(t)) \\
            & \quad := \left( \sigmaden(t) + \chi \left( 1 - \left( \phiden(t) - \frac{\delta}{\gamma} \muden(t) \right) \right) - (\muden(t) + \Pi^n(\Psipe'(\phiden(t)))) \right).
        \end{align*}
        It suffices to solve \eqref{eq:varform:phi_eps_n}--\eqref{eq:varform:ic_eps_n} only for $v=w_j$, $j=1,\dots,n$.
        Then, by inserting the expressions \eqref{def:discr_sols} in \eqref{eq:varform:phi_eps_n}--\eqref{eq:varform:ic_eps_n} with $v = w_j$, $j=1,\dots,n$, and by computing the $H$-inner products, one can easily see that \eqref{eq:varform:phi_eps_n}--\eqref{eq:varform:ic_eps_n} becomes a system of $3n$ nonlinear ODEs in the unknowns $a_j^n, c_j^n, d_j^n: [0,T] \to \R$.
        Consequently, owing to the continuity of $b_\eps$, $\Psipe$, $\Psi_-$ and $P_\eps$, by the Cauchy--Peano Theorem there exist local solutions $a_j^n, c_j^n, d_j^n \in C^1([0, t_n))$, $j = 1, \dots, n$, for some $t_n \in [0,T]$.
        Then, the local solutions $\phiden, \muden, \sigmaden \in C^1([0,t_n);{W_n})$ to \eqref{eq:varform:phi_eps_n}--\eqref{eq:varform:ic_eps_n} are well-defined.

        We now obtain some \emph{a priori} estimates on $\phiden, \muden, \sigmaden$, with constants independent of $n$, that will allow us to extend the solution to $[0,T]$ and then pass to the limit as $n \to \infty$.
        To do this, it is useful to formally rewrite \eqref{eq:mueps} as
        \begin{equation}
            \label{eq:mueps_prime}
            \mude = - \gamma \Delta \left( \phide - \dg \mude \right) + \Psi'_- \left( \phide - \dg \mude \right) - \chi \sigmade,
        \end{equation}
        which, in terms of the discrete variational formulation above, becomes
        \begin{equation}
            \label{eq:varform:mu_eps_n_prime}
            \begin{split}
            (\muden(t), v)_H
            & = \gamma \left( \nabla \left( \phiden(t) - \dg \muden(t) \right), \nabla v \right)_{\!\! H} \\
            & \quad + \left( \Psi'_-\left( \phiden(t) - \dg \muden(t) \right) - \chi \sigmaden(t), v \right)_{\!\! H},
            \end{split}
        \end{equation}
        for any $v \in W_n$ and for any $t \in (0,t_n)$.

        \textsc{Energy estimate.}
        To get the main energy estimate, for any given $t \in (0,t_n)$, we take $v = \muden + \Pi^n(\Psipe'(\phiden))$ in \eqref{eq:varform:phi_eps_n}, $v = - \partial_t \left( \phiden - \dg \muden \right)$ in \eqref{eq:varform:mu_eps_n_prime} and $v = \sigmaden + \chi \left( 1 - \left(\phiden - \dg \muden \right) \right)$ in \eqref{eq:varform:sigma_eps_n}.
        We stress that all the above test functions are allowed, since their expressions make sense in $W_n$ for any $t \in (0,t_n)$.
        Then, by summing up the equations, after some cancellations, we obtain that
        \begin{equation}
		\label{eq:energyest_eps_n}
		    \begin{split}
    		& \ddt \Bigg( \int_\Omega \left( \Psipe(\phiden) + \Psi_-\left(\phiden - \dg \muden \right) \right) \, \de x
    		+ \mezzo \dg \int_\Omega \abs{\muden}^2 \, \de x \\
    		& \qquad
            + \frac{\gamma}{2} \int_\Omega \abs*{\nabla \left(\phiden - \dg \muden \right)}^2 \,\de x
    		+ \mezzo \int_\Omega \abs{\sigmaden}^2 \, \de x
    		- \int_\Omega \chi \sigmaden \left( 1 - \left(\phiden - \dg \muden \right) \right) \, \de x \Bigg) \\
    		& \quad
    		+ \int_\Omega b_\eps(\phiden) \abs*{\nabla (\muden + \Pi^n(\Psipe'(\phiden)))}^2 \, \de x
    		+ \int_\Omega \abs*{\nabla \left(\sigmaden +  \chi \left( 1 - \left(\phiden - \dg \muden \right) \right) \right)}^2 \, \de x
    		\\
    		& \quad
    		+ \int_\Omega \left[ \sqrt{P_\eps(\phiden)} \left( \sigmaden + \chi \left( 1 - \left( \phiden - \frac{\delta}{\gamma} \muden \right) \right) - (\muden + \Pi^n(\Psipe'(\phiden))) \right) \right]^2 \, \de x
    		= 0,
    		\end{split}
	   \end{equation}
       where the term differentiated with respect to time is exactly the free energy $\mathcal{E}_\delta(\phiden, \sigmaden)$, with $\Psi_+$ replaced by $\Psipe$ (see \eqref{eq:freeenergy}).
        Differently from \cite{PP2021}, due to the chemotaxis and the reaction terms, we need to have some additional coercivity properties on the free energy.
        In fact, we invoke \eqref{ass:coerc} to recall that
        \[
            \Psi_-\left(\phiden - \frac{\delta}{\gamma} \muden \right) \ge c_1 \left(\phiden - \frac{\delta}{\gamma} \muden \right)^2 - c_2 \quad \hbox{in $Q_T$.}
        \]
        Then, by the Cauchy--Schwarz and the Young inequalities, we can estimate the chemotactic term from below as
        \begin{align*}
            & - \int_\Omega \chi \sigmaden \left( 1 - \left(\phiden - \dg \muden \right) \right) \, \de x \\
            & \quad \ge
            - \left( \frac12 - \alpha \right) \int_\Omega \abs{\sigmaden}^2 \, \de x
            - \frac{\chi^2}{2 - 4\alpha} \int_\Omega \left( 1 - \left( \phiden - \dg \muden \right) \right)^2 \, \de x\\
            & \quad \ge
            - \left( \frac12 - \alpha \right) \int_\Omega \abs{\sigmaden}^2 \, \de x
            - \frac{\chi^2}{1 - 2 \alpha} \abs{\Omega}
            - \frac{\chi^2}{1 - 2\alpha} \int_\Omega \left( \phiden - \dg \muden \right)^2 \, \de x,
        \end{align*}
        for any $\alpha \in (0, 1/2)$.
        Therefore, with hypothesis \ref{ass:chi}, one can find $\alpha \in (0, 1/2)$ such that
        \[
            \beta := c_1 - \frac{\chi^2}{1 - 2\alpha} > 0.
        \]
        Hence, we get
        \begin{align*}
            \mathcal{E}_{\delta, \eps}(\phiden, \sigmaden) & \ge
            \beta \int_\Omega \abs*{\phiden - \dg \muden}^2 \de x
            + \mezzo \dg \int_\Omega \abs{\muden}^2 \, \de x \\
            & \quad
            + \frac{\gamma}{2} \int_\Omega \abs*{\nabla \left(\phiden - \dg \muden \right)}^2 \,\de x
            + \alpha \int_\Omega \abs{\sigmaden}^2 \, \de x - \left( c_2 + \frac{\chi^2}{1 - 2 \alpha} \right) \abs{\Omega},
        \end{align*}
        where both $\alpha$ and $\beta$ are positive constants independent of $n$, $\eps$, $\delta$ and $T$.
        Thus, integrating \eqref{eq:energyest_eps_n} in $(0,t)$, for any given $t \in (0,t_n)$, and using the above lower bound on the free energy, we deduce that
        \begin{equation}
            \label{eq:energyineq_eps_n}
            \begin{split}
            & \beta \int_\Omega \abs*{\phiden(t) - \dg \muden(t)}^2 \de x
            + \mezzo \dg \int_\Omega \abs{\muden(t)}^2 \, \de x \\
            & \qquad
            + \frac{\gamma}{2} \int_\Omega \abs*{\nabla \left(\phiden(t) - \dg \muden(t) \right)}^2 \,\de x
            + \alpha \int_\Omega \abs{\sigmaden(t)}^2 \, \de x \\
            & \qquad
            + \int_0^t \int_\Omega b_\eps(\phiden) \abs*{\nabla (\muden + \Pi^n(\Psipe'(\phiden)))}^2 \, \de x \, \de s \\
            & \qquad
            + \int_0^t \int_\Omega \abs*{\nabla \left(\sigmaden + \chi \left( 1 - \left(\phiden - \dg \muden \right) \right) \right)}^2 \, \de x \, \de s
            \\
            & \qquad
            + \int_0^t \int_\Omega \left[ \sqrt{P_\eps(\phiden)} \left( \sigmaden + \chi \left( 1 - \left( \phiden - \frac{\delta}{\gamma} \muden \right) \right) - (\muden + \Pi^n((\Psipe'(\phiden))) \right) \right]^2 \, \de x \, \de s \\
            & \quad \le
            \mathcal{E}_{\delta, \eps}(\Pi^n\phi_0, \Pi^n\sigma_0) + C \abs{\Omega},
            \end{split}
        \end{equation}
        where $C>0$ is a constant independent of $n$, $\eps$, $\delta$ and $T$.
        We now comment on the term regarding the initial energy in the discretised initial data, namely,
        \begin{align*}
            & \mathcal{E}_{\delta, \eps}(\Pi^n\phi_0, \Pi^n\sigma_0) =
            \int_\Omega \left( \Psipe(\Pi^n \phi_0) + \Psi_-\left(\Pi^n \phi_0 - \dg \muden(0) \right) \right) \, \de x \\
    		& \quad
            + \mezzo \dg \int_\Omega \abs{\muden(0)}^2 \, \de x
            + \frac{\gamma}{2} \int_\Omega \abs*{\nabla \left(\Pi^n \phi_0 - \dg \muden(0) \right)}^2 \,\de x \\
    		& \quad
            + \mezzo \int_\Omega \abs{\Pi^n \sigma_0}^2 \, \de x
    		- \int_\Omega \chi \Pi^n \sigma_0 \left( 1 - \left(\Pi^n \sigma_0 - \dg \muden(0) \right) \right) \, \de x.
        \end{align*}
        First of all, we observe that, if the initial data satisfy \ref{ass:iniz}, then by Remark \ref{rmk:initial_data} the initial energy $\mathcal{E}_\delta(\phi_0, \sigma_0)$ is bounded.
        Next, we aim to replicate the argument of Remark \ref{rmk:initial_data} at the discrete level to find $\muden(0) \in W_n$ through the variational formulation \eqref{eq:varform:mu_eps_n_prime} in $W_n$, starting from $\Pi^n \phi_0$ and $\Pi^n \sigma_0$.
        Indeed, given $\Pi^n \phi_0 \in W_n$ and $\Pi^n \sigma_0 \in W_n$, $\Pi^n \phi_0 - \delta \muden(0) \in W_n$ can be determined as the unique solution to the following discretised elliptic system, which can be deduced from \eqref{eq:varform:mu_eps_n_prime} at time $t=0$ by multiplying by $\delta/\gamma$ and by adding and subtracting $\Pi^n \phi_0$, namely
        \begin{align*}
            & \delta \left( \nabla \left(\Pi^n \phi_0 - \dg \muden(0) \right), \nabla v \right)_{\!\! H}
            + \left( \left( \Pi^n \phi_0 - \dg \muden(0) \right), v \right)_{\!\! H} \\
            & \quad = (\Pi^n \phi_0, v)_H
            - \dg \left( \Psi'_- \left(\Pi^n \phi_0 - \dg \muden(0)\right), v \right)_{\!\! H} + \dg \chi (\Pi^n \sigma_0, v)_H,
        \end{align*}
        for any $v \in W_n$.
        Then, one uniquely finds $\Pi^n \phi_0 - \dg \muden(0) \in W_n$, and, since $\Pi^n \phi_0 \in W_n$, deduces that $\muden(0) \in W_n$.
        By the properties of $\Pi^n$, it follows that $\Pi^n \phi_0 \to \phi_0$ in $H$, $\Pi^n \sigma_0 \to \sigma_0$ in $H$, $\muden(0) = \Pi^n \mude(0) \to \mude(0)$ in $H$ and $\Pi^n \phi_0 - \dg \muden(0) = \Pi^n \left( \phi_0 - \dg \mude(0) \right) \to \phi_0 - \dg \mude(0)$ in $V$, where $\mude(0)$ is the corresponding initial data for the regularised chemical potential, found as in Remark \ref{rmk:initial_data}.
        In particular, the last convergence follows from the fact that $\Pi^n \phi_0 - \dg \muden(0)$ is the solution to the elliptic equation above, together with the convergences $\Pi^n \phi_0 \to \phi_0$ and $\Pi^n \sigma_0 \to \sigma_0$ in $H$.
        Moreover, by \eqref{eq:reg_pot}, $\Psipe$ has at most quadratic growth, so there exists $C_\eps > 0$ such that
        \[
            \norm{\Psipe(\Pi^n \phi_0)}_{\Lx 1}
            \le C_\eps (\norm{\Pi^n \phi_0}^2_H + 1)
            \le C_\eps (\norm{\phi_0}^2_H + 1),
        \]
        owing also to the convergence $\Pi^n \phi_0 \to \phi_0$ in $H$.
        By the continuity of $\Psipe$ and the almost everywhere convergence $\Pi^n \phi_0 \to \phi_0$ (up to a non-relabelled subsequence), we also have that $\Psipe(\Pi^n \phi_0) \to \Psipe(\phi_0)$ almost everywhere in $\Omega$.
        Then, by the Dominated Convergence Theorem, it follows that $\Psipe(\Pi^n \phi_0) \to \Psipe(\phi_0)$ in $\Lx 1$.
        A similar argument also provides the convergence $\Psi_- \left( \Pi^n\phi_0 - \dg \muden(0) \right) \to \Psi_-\left( \phi_0 - \dg \mude(0) \right)$ in $\Lx 1$, as $\Psi_-$ also has at most quadratic growth.
        Consequently, by the convergences above, we deduce that
        \begin{equation}
        \label{eq:conv_initialenergy_n}
            \mathcal{E}_{\delta, \eps}(\Pi^n \phi_0, \Pi^n \sigma_0) \to \mathcal{E}_{\delta, \eps}(\phi_0, \sigma_0) \quad \hbox{as $n \to \infty$.}
        \end{equation}
        Moreover, since $\Psipe \le \Psi_+$ in $[0,1]$ by convexity and its definition \eqref{eq:reg_pot}, and $0 \le \phi_0 < 1$ almost everywhere in $\Omega$ with $\Psi_+(\phi_0) \in \Lx 1$ by \ref{ass:iniz}, we can bound
        \[
            \mathcal{E}_{\delta, \eps}(\phi_0, \sigma_0) \le \mathcal{E}_{\delta}(\phi_0, \sigma_0),
        \]
        for $\eps$ small enough.
        This means that the right-hand side of \eqref{eq:energyineq_eps_n} can be replaced by a uniform bound depending only on $\mathcal{E}_\delta(\phi_0, \sigma_0)$ and the structural parameters of the system, but independent of $n$, $\eps$ and $T$. This has two important consequences.
        First, the local solutions $\phiden, \muden, \sigmaden \in C^1([0, t_n); W_n)$ can actually be extended to global solutions $\phiden, \muden, \sigmaden \in C^1([0, T]; W_n)$ by a standard continuation argument.
        Then, \eqref{eq:energyineq_eps_n} provides the following uniform estimates:
        \begin{equation}
            \label{eq:unifbounds_eps_n}
            \begin{split}
                & \norm*{\phiden - \dg \muden}_{\LT \infty V} \le C, \\
                & \sqrt{\delta} \norm{\muden}_{\LT \infty H} \le C, \\
                & \norm{\sigmaden}_{\LT \infty H} \le C, \\
                & \norm*{\sqrt{b_\eps(\phiden)} \nabla (\muden + \Pi^n(\Psipe'(\phiden)))}_{\LT 2 H} \le C, \\
                & \norm*{\nabla \left( \sigmaden + \chi \left( 1 - \left(\phiden - \dg \muden \right) \right) \right)}_{\LT 2 H} \le C, \\
                & \norm*{\sqrt{P_\eps(\phiden)} \, R^n_{\delta, \eps}(\phiden, \muden, \sigmaden)}_{\LT 2 H} \le C,
            \end{split}
        \end{equation}
        with  $C > 0$ independent of $n$, $\eps$ and $T$, and depending on $\delta$ only through $\mathcal{E}_\delta(\phi_0, \sigma_0)$.
        From the first and fifth bound in \eqref{eq:unifbounds_eps_n} it also follows that
        \begin{equation}
            \label{eq:bound_sigma_l2v_n}
            \norm{\sigmaden}_{\LT 2 V} \le C.
        \end{equation}

        \textsc{Gradient estimate on $\phiden$.}
        Following the ideas used in the proof of \cite[Theorem 1]{PP2021}, we now deduce a uniform estimate on $\nabla \phiden$.
        First, by factoring some terms and recalling that, for the regularised potential, $\nabla \Psipe'(\phiden) = \Psipe''(\phiden) \nabla \phiden$ by the chain rule, we observe that
        \begin{align*}
            & \int_\Omega \abs*{\nabla \left(\phiden + \dg \Psipe'(\phiden) \right)}^2 \, \de x
            = \int_\Omega \, \abs*{\nabla \Psipe'(\phiden) \left( \frac{1}{\Psipe''(\phiden)} + \dg \right)}^2 \, \de x \\
            & \quad
            \ge \inf_{s \in \R} \left( \frac{1}{\Psipe''(s)} + \dg \right)^2 \int_\Omega \abs{\nabla \Psipe'(\phiden)}^2 \, \de x \\
            & \quad
            \ge \left( \theta_\eps^2 + \left( \dg \right)^2  \right) \int_\Omega \abs{\nabla \Psipe'(\phiden)}^2 \, \de x,
        \end{align*}
        for some $\theta_\eps > 0$.
        Here, we crucially used the fact that, by \eqref{eq:reg_pot}--\eqref{ass:Psi_eps}, for any fixed $\eps > 0$ the regularised potential $\Psipe$ is a strictly convex function with bounded second derivative.
        Next, reading the inequality in the opposite direction, we add and subtract some terms to obtain that
        \begin{align*}
            & \left( \theta_\eps^2 + \left( \dg \right)^2  \right) \int_\Omega \abs{\nabla \Psipe'(\phiden)}^2 \, \de x
            \le \int_\Omega \abs*{\nabla \left(\phiden + \dg \Psipe'(\phiden) \right)}^2 \, \de x \\
            & \quad \le \int_\Omega \abs*{\nabla \left( \phiden - \dg \muden \right)}^2 \, \de x
            + \left( \dg \right)^2 \int_\Omega \abs*{\nabla \left( \muden + \Pi^n (\Psipe'(\phiden)) \right)}^2 \, \de x \\
            & \qquad
            + \left( \dg \right)^2 \int_\Omega \abs*{\nabla \left( \Psipe'(\phiden) - \Pi^n (\Psipe'(\phiden)) \right)}^2 \, \de x \\
            & \quad \le \int_\Omega \abs*{\nabla \left( \phiden - \dg \muden \right)}^2 \, \de x
            + \left( \dg \right)^2 \int_\Omega \abs*{\nabla \left( \muden + \Pi^n (\Psipe'(\phiden)) \right)}^2 \, \de x \\
            & \qquad
            + \left( \dg \right)^2 \int_\Omega \abs{\nabla \Psipe'(\phiden)}^2 \, \de x,
        \end{align*}
        where the last inequality follows from the properties of the $L^2$ projection operator $\Pi^n$ onto the discrete spaces $W_n$.
        Indeed, such projection operators not only satisfy the standard inequality $\norm{v - \Pi^n v}_H \le \norm{v}_H$ for any $v \in H$, but also $\norm{\nabla(v - \Pi^n v)}_H \le \norm{\nabla v}_H$ for any $v \in V$ due to the fact that $W_n$ is generated by eigenvectors of the operator $\mathcal{N}$.
        Therefore, by cancelling the equal terms on both sides and integrating on $(0,T)$, we get that
        \begin{align*}
            & \theta_\eps^2 \int_0^T \int_\Omega \abs{\nabla \Psipe'(\phiden)}^2 \, \de x \, \de t \\
            & \quad \le
            \int_0^T \int_\Omega \abs*{\nabla \left( \phiden - \dg \muden \right)}^2 \, \de x \, \de t
            + \left( \dg \right)^2 \int_0^T \int_\Omega \abs*{\nabla \left( \muden + \Pi^n (\Psipe'(\phiden)) \right)}^2 \, \de x \, \de t,
        \end{align*}
        which, owing to the uniform bounds \eqref{eq:unifbounds_eps_n} and to the fact that the approximated mobility $b_\eps$ is bounded below (cf. \eqref{ass:b_eps}), imply that
        \begin{equation*}
            \label{eq:bound_Psipe_n}
            \theta_\eps^2 \int_0^T \int_\Omega \abs{\nabla \Psipe'(\phiden)}^2 \, \de x \, \de t \le C_T.      \end{equation*}
        Here $C_T > 0$ may depend also on $T$, since on $\phiden - \dg \muden$ is bounded in the stronger space $\LT \infty V$.
        Consequently, by computing again the gradient and exploiting the regularity of $\Psipe$, as well as the non-degeneracy of its second derivative (cf. \eqref{ass:Psi_eps}), we deduce that
        \begin{align*}
            \theta_\eps^2 \bar{c}_\eps \int_0^T \int_\Omega \abs{\nabla \phiden}^2 \, \de x \, \de t
            \le \theta_\eps^2 \left( \inf_{s \in \R} \Psipe''(s) \right) \int_0^T \int_\Omega \abs{\nabla \phiden}^2 \, \de x \, \de t
            \le C_T.
        \end{align*}
        Hence, we get the uniform bound
        \begin{equation}
        \label{eq:bound_nablaphi_n}
            \norm{\nabla \phiden}_{\LT 2 H} \le C_{T,\eps}.
        \end{equation}
        Moreover, by comparing \eqref{eq:bound_nablaphi_n} with the first uniform estimate in \eqref{eq:unifbounds_eps_n}, we also infer that
        \begin{equation}
        \label{eq:bound_nablamu_n}
            \sqrt{\delta} \norm{\nabla \muden}_{\LT 2 H} \le C_{T, \eps}.
        \end{equation}

        \textsc{Mean value estimate.}
        We now aim to find some uniform estimate on the mean value of $\phiden$. 
        Hence, we choose the constant function $v = 1/\abs{\Omega}$ in \eqref{eq:varform:phi_eps_n}, which is possible due to the fact that the first eigenfunction of $\Ncal$ is constant.
        This yields
        \[
            \partial_t \overline{\phiden} = \frac{1}{\abs{\Omega}} \int_\Omega P_\eps(\phiden) R^n_{\delta,\eps}(\phiden, \muden, \sigmaden) \, \de x.
        \]
        Then, by integrating on $(0,t)$, for any $t \in (0,T)$, and employing the Cauchy--Schwarz and the Young  inequalities, we deduce that
        \begin{align*}
            \overline{\phiden}(t) &
            = \overline{\Pi^n \phi_0}
            + \frac{1}{\abs{\Omega}} \int_0^t \int_\Omega \sqrt{P_\eps(\phiden)} \sqrt{P_\eps(\phiden)} R^n_{\delta,\eps}(\phiden, \muden, \sigmaden) \, \de x \\
            & \le \norm{\phi_0}_{\Lx 1}
            + \frac{1}{2 \abs{\Omega}} \int_0^T \int_\Omega P_\eps(\phiden) \, \de x \, \de t \\
            & \quad + \frac{1}{2 \abs{\Omega}} \int_0^T \int_\Omega P_\eps(\phiden) R^n_{\delta,\eps}(\phiden, \muden, \sigmaden)^2 \, \de x \, \de t \\
            & \le \norm{\phi_0}_{\Lx 1} + \norm{P_\eps}_\infty T + C
            \le C_T,
        \end{align*}
        owing to \ref{ass:iniz}, the uniform boundedness of $P_\eps$ (cf. \eqref{eq:reg_prol}) and the last bound in \eqref{eq:unifbounds_eps_n}.
        Consequently, we deduce that
        \begin{equation}
            \label{eq:bound_phi_mean_n}
            \norm{\overline{\phiden}}_{\Lt \infty} \le C_T.
        \end{equation}
        By the Poincar\'e--Wirtinger inequality, \eqref{eq:bound_nablaphi_n} and \eqref{eq:bound_phi_mean_n} also imply that
        \begin{equation}
            \label{eq:bound_phi_l2v_n}
            \norm{\phiden}_{\LT 2 V} \le C_{T, \eps}.
        \end{equation}

        \textsc{A higher order estimate.}
        We now choose $v = - \Delta \left( \phiden - \dg \muden \right)$ in \eqref{eq:varform:mu_eps_n_prime}, which is possible since $\phiden$ and $\muden$ are written in terms of eigenvectors of $\Ncal$.
        Then, by integrating by parts and applying the Cauchy--Schwarz and the Young inequalities, as well as hypothesis \ref{ass:psi} on $\Psi_-'$, we obtain that
        \begin{align*}
            & \gamma \int_\Omega \abs*{\Delta \left(\phiden - \dg \muden \right)}^2 \, \de x \\
            & \quad = \int_\Omega \muden \, \Delta \left(\phiden - \dg \muden \right) \, \de x
            + \int_\Omega \Psi_-'\left(\phiden - \dg \muden \right) \, \Delta \left(\phiden - \dg \muden \right) \, \de x \\
            & \qquad - \chi \int_\Omega \sigmaden \, \Delta \left(\phiden - \dg \muden \right) \, \de x \\
            & \quad \le \frac{\gamma}{2} \int_\Omega \abs*{\Delta \left(\phiden - \dg \muden \right)}^2 \, \de x
            + C \int_\Omega \abs{\muden}^2 \, \de x \\
            & \qquad + C \int_\Omega \abs*{\phiden - \dg \muden}^2 \, \de x
            + C \int_\Omega \abs{\sigmaden}^2_H \, \de x + C.
        \end{align*}
        Then, by integrating on $(0,T)$ and recalling the uniform bounds \eqref{eq:unifbounds_eps_n}, we deduce that
        \[
            \norm*{\Delta \left(\phiden - \dg \muden \right)}_{\LT 2 H} \le C_\delta,
        \]
        which, by elliptic regularity theory and \eqref{eq:bceps}, readily implies
        \begin{equation}
            \label{eq:bound_phidgmu_h2_n}
            \norm*{\phiden - \dg \muden}_{\LT 2 {\Hx 2}} \le C_\delta.
        \end{equation}
        Note that the constant $C_\delta > 0$ now depends on $\delta$ since the second bound in \eqref{eq:unifbounds_eps_n} was used.

        \textsc{Estimates on the time derivatives.}
        Lastly, in order to use compactness theorems, we need some uniform estimates on the time derivatives $\partial_t \phiden$ and $\partial_t \sigmaden$, as well as $\partial_t \left(\phiden - \dg \muden \right)$, in the dual space $V^*$.
        Hence, for any test function $v \in V$, we decompose $v = v_1 + v_2$ with $v_1 \in W_n$ and $v_2 \in W_n^\perp$.
        Then, owing to the fact that $\partial_t \phiden \in W_n$ for almost any $t \in (0,T)$, we can insert $v_1$ in \eqref{eq:varform:phi_eps_n} and employ H\"older's inequality and Sobolev embeddings to obtain
    	\begin{align*}
    		& \duality{\partial_t \phiden, v}_V
            = \duality{\partial_t \phiden, v_1}_V \\
    		& \quad =
    		\int_\Omega b_\eps(\phiden) \nabla (\muden + \Pi^n(\Psipe'(\phiden))) \cdot \nabla v_1 \, \de x
    		+ \int_\Omega P_\eps(\phiden) R^n_{\delta,\eps}(\phiden, \muden, \sigmaden) \, v_1 \, \de x \\
    		& \quad \le \norm*{\sqrt{b_\eps(\phiden)}}_{\Lx\infty} \norm*{\sqrt{b_\eps(\phiden)} \nabla (\muden + \Pi^n(\Psipe'(\phiden)))}_H \norm{\nabla v_1}_H \\
    		& \qquad + \norm*{\sqrt{P_\eps(\phiden)}}_{\Lx3} \norm*{\sqrt{P_\eps(\phiden)} R^n_{\delta, \eps}(\phiden, \muden, \sigmaden)}_H \norm{v_1}_{\Lx 6} \\
    		& \quad \le C \left( \norm*{\sqrt{b_\eps(\phiden)} \nabla (\muden + \Pi^n(\Psipe'(\phiden)))}_H + \norm*{\sqrt{P_\eps(\phiden)} R^n_{\delta, \eps}(\phiden, \muden, \sigmaden)}_H  \right) \norm{v}_V,
    	\end{align*}
    	since both $b_\eps$ and $P_\eps$ are globally bounded functions belonging to $C^0(\R)$ uniformly in $\eps$.
    	Then, we deduce that
    	\begin{align*}
    		& \int_0^T \norm{\partial_t \phiden}^2_{\LT 2 {V^*}} \, \de t
            \le C \bigg( \int_0^T \norm*{\sqrt{b_\eps(\phiden)} \nabla (\muden + \Pi^n(\Psipe'(\phiden)))}^2_H \, \de t \\
            & \qquad + \int_0^T \norm*{\sqrt{P_\eps(\phiden)} R^n_{\delta, \eps}(\phiden, \muden, \sigmaden)}^2_H \, \de t \bigg),
    	\end{align*}
    	which, by the uniform bounds \eqref{eq:unifbounds_eps_n}, immediately yields
    	\begin{equation}
    		\label{eq:bound_dtphi_n}
    			\norm{\phiden}_{\HT 1 {V^*}} \le C,
    	\end{equation}
    	uniformly in $\delta$, $\eps$, $n$ and $T$.
    	Arguing similarly and testing \eqref{eq:varform:sigma_eps_n} with any $v \in V$, one can also deduce that
        \begin{align*}
    		& \int_0^T \norm{\partial_t \sigmaden}^2_{\LT 2 {V^*}} \, \de t
            \le C \Bigg( \int_0^T \norm*{\nabla \left( \sigmaden + \chi \left( 1 - \left( \phiden - \dg \muden \right) \right) \right)}^2_H \, \de t \\
            & \qquad + \int_0^T \norm*{\sqrt{P_\eps(\phiden)} R^n_{\delta, \eps}(\phiden, \muden, \sigmaden)}^2_H \, \de t \Bigg),
    	\end{align*}
        which, again by \eqref{eq:unifbounds_eps_n}, implies that
    	\begin{equation}
    		\label{eq:bound_dtsigma_n}
    			\norm{\sigmaden}_{\HT 1 {V^*}} \le C,
    	\end{equation}
        uniformly in $\delta$, $\eps$, $n$ and $T$.

        Finally, we can also perform an additional estimate on $\partial_t (\phiden - \dg \muden)$, as in \cite{PP2021}.
        Here we use a more rigorous approach with respect to that of \cite{PP2021}, which leads to changing the smallness assumption in \ref{ass:psi_small}.
    	To do this, we recall \eqref{eq:mueps_prime} and,  multiplying by $\dg$ and adding $\phiden$ on both sides, we get
    	\begin{equation}
    		\label{eq:mueps_prime2}
    		\phiden - \dg \muden
    		= \phiden
    		- \delta \Delta \left( \phiden - \dg \muden \right)
    		- \dg \Psi'_-\left( \phiden - \dg \muden \right)
    		+ \dg \chi \sigmaden.
    	\end{equation}
    	Then, we differentiate \eqref{eq:mueps_prime2} with respect to time. This yields
    	\begin{align*}
    		& \partial_t \left( \phiden - \dg \muden \right)
    		- \delta \Delta \left( \partial_t \left( \phiden - \dg \muden \right) \right) \\
    		& \quad = \partial_t \phiden
    		- \dg \Psi''_-\left( \phiden - \dg \muden \right) \partial_t \left( \phiden - \dg \muden \right)
    		+ \dg \chi \partial_t \sigmaden.
    	\end{align*}
    	Thus, setting
    	\[
    		U := \partial_t \left( \phiden - \dg \muden \right),
    	\]
    	  we obtain the following elliptic equation for $U$:
    	\begin{equation}
    	\label{eq:U_eps_n}
    		- \delta \Delta U
    		+ \left( 1 + \dg \Psi''_-\left( \phiden - \dg \muden \right) \right) U
    		= \partial_t \phiden
    		+ \dg \chi \partial_t \sigmaden.
    	\end{equation}
        We now observe that all the steps above can be rigorously repeated also at the level of the discrete variational formulation.
        Indeed, one can start from \eqref{eq:varform:mu_eps_n_prime}, multiply by $\dg$, add $(\phiden, v)_H$ on both sides and, then, differentiate in time.
        We stress that time-differentiation is allowed in the discretised framework as it only affects the coefficients $a_j^n, c_j^n, d_j^n \in C^1([0,T])$, so that the functions $\partial_t \phiden$, $\partial_t \muden$ and $\partial_t \sigmaden$ are still well-defined with values in $W_n$.
        Consequently, the same $U$ satisfies the variational formulation
        \begin{equation}
    	\label{eq:varform:U_eps_n}
    		\delta ( \nabla U, \nabla v)_H
    		+ \left( \left( 1 + \dg \Psi''_-\left( \phiden - \dg \muden \right) \right) U, v \right)_{\!\! H}
    		= \left( \partial_t \phiden
    		+ \dg \chi \partial_t \sigmaden, v \right)_{\!\! H},
    	\end{equation}
        for any $v \in W_n$.
    	Observe that the right-hand side $\partial_t \phiden
    	+ \dg \chi \partial_t \sigmaden$, is uniformly bounded in $\LT 2 {V^*}$ by \eqref{eq:bound_dtphi_n}--\eqref{eq:bound_dtsigma_n}.
        With the idea of exploiting this bound, we would like to test \eqref{eq:varform:U_eps_n} by $v = \Ncal^{-1} U$, which is still an admissible test function since $U \in W_n$ for almost any $t \in (0,T)$ and $W_n$ is generated by the eigenvectors of $\Ncal$.
        To do this, we add and subtract $\delta U$ on the left-hand side of \eqref{eq:varform:U_eps_n}. This gives
        \begin{align*}
            \delta (\Ncal U, v)_H + (1 - \delta) (U, v)_H
            = - \dg \left( \Psi''_-\left( \phiden - \dg \muden \right), U \right)_{\!\! H}
            + \left( \partial_t \phiden
    		+ \dg \chi \partial_t \sigmaden, v \right)_{\!\! H},
        \end{align*}
        for any $v \in W_n$.
        Then, choosing $v = \Ncal^{-1} U$ in the equation above and setting $f := \partial_t \phiden
    	+ \dg \chi \partial_t \sigmaden$, by means of standard inequalities and using the properties of the operator $\Ncal$, we infer that
    	\begin{align*}
    		& \delta \norm{U}^2_H + (1 - \delta) \norm{U}^2_{V^*} =
    		\duality{f, \Ncal^{-1}U}_{V}
    		- \duality*{ \dg \Psi''_-\left( \phiden - \dg \muden \right) U, \Ncal^{-1}U}_{V} \\
    		& \quad \le
    		\norm{f}_{V^*} \norm{\Ncal^{-1}U}_{V}
    		+ \dg \norm*{\Psi''_-\left( \phiden - \dg \muden \right) U}_{V^*} \norm{\Ncal^{-1}U}_{V} \\
    		& \quad \le
    		\norm{f}_{V^*} \norm{U}_{V^*}
    		+ \dg \norm*{\Psi''_-\left( \phiden - \dg \muden \right) U}_{H} \norm{U}_{V^*} \\
    		& \quad \le
    		\frac14 \norm{U}^2_{V^*} + \norm{f}^2_{V^*}
    		+ \dg \norm{\Psi''_-}_\infty \norm{U}_{H} \norm{U}_{V^*} \\
    		& \quad \le \mezzo \norm{U}^2_{V^*}
    		+ \norm{f}^2_{V^*}
    		+ \frac{\delta^2}{\gamma^2} \norm{\Psi''_-}^2_\infty \norm{U}^2_{H}.
    	\end{align*}
    	Hence, by also integrating on $(0,T)$, we obtain that
    	\begin{equation}
    		\label{eq:estimate_U_n}
    		\delta \underbrace{\left( 1 - \frac{\delta}{\gamma^2} \norm{\Psi''_-}^2_\infty \right)}_{> \, 0} \int_0^T \norm{U}^2_{H} \, \de t
    		+ \left(\mezzo - \delta \right) \int_0^T \norm{U}^2_{V^*} \, \de t
    		\le \int_0^T \norm{f}^2_{V^*} \, \de t.
    	\end{equation}
    	Then, by  \ref{ass:psi_small} and estimates \eqref{eq:bound_dtphi_n}--\eqref{eq:bound_dtsigma_n}, we find the uniform bound
    	\begin{equation}
    		\label{eq:bound_dtphimu_n}
    		\norm*{\partial_t \left( \phiden - \dg \muden \right)}_{\LT 2 {V^*}} \le C,
    	\end{equation}
    	where $C > 0$ is independent of $\delta$, $\eps$ and $n$.

        \textsc{Passage to the limit.}
        We now pass to the limit in the discretisation as $n \to + \infty$ to recover a weak solution to the regularised problem \eqref{eq:phieps}--\eqref{eq:iceps} (see \eqref{eq:varform:phi_eps}--\eqref{eq:varform:sigma_eps}).
        If not further specified, all the convergences below have to be intended as $n \to \infty$.
        By Banach--Alaoglu's theorem, up to a non-relabelled subsequence, the uniform bounds \eqref{eq:bound_sigma_l2v_n}, 
        \eqref{eq:bound_nablamu_n},
        \eqref{eq:bound_phi_l2v_n}, and \eqref{eq:bound_phidgmu_h2_n} imply the weak convergences
        \begin{align}
            & \phiden \weak \phide \quad \hbox{weakly in $\LT 2 V$,} \label{eq:conv_phi_l2v_n} \\
            & \muden \weak \mude \quad \hbox{weakly in $\LT 2 V$,} \label{eq:conv_mu_l2v_n} \\
            & \sigmaden \weakstar \sigmade \quad \hbox{weakly star in $\LT \infty H \cap \LT 2 V$,} \label{eq:conv_sigma_l2v_n} \\
            & \phiden - \dg \muden \weakstar \phiden - \dg \muden \quad \hbox{weakly star in $\LT \infty V \cap \LT 2 {\Hx2}$.} \label{eq:conv_phidgmu_l2w_n}
        \end{align}
        Similarly, by \eqref{eq:bound_dtphi_n}, \eqref{eq:bound_dtsigma_n} and \eqref{eq:bound_dtphimu_n}, we deduce
        \begin{align}
            & \phiden \weak \phide \quad \hbox{weakly in $\HT 1 {V^*}$,} \label{eq:conv_phi_h1vs_n} \\
            & \sigmaden \weak \sigmade \quad \hbox{weakly in $\HT 1 {V^*}$,} \label{eq:conv_sigma_h1vs_n} \\
            & \phiden - \dg \muden \weak \phide - \dg \mude \quad \hbox{weakly in $\HT 1 {V^*}$.} \label{eq:conv_phidgmu_h1vs_n}
        \end{align}
        Then, known compact embeddings  (see \cite[Section 8, Corollary 4]{S1986}), \eqref{eq:conv_phi_l2v_n}--\eqref{eq:conv_phidgmu_l2w_n} and \eqref{eq:conv_phi_h1vs_n}--\eqref{eq:conv_phidgmu_h1vs_n} entail the following strong convergences:
        \begin{align}
            & \phiden \to \phide \quad \hbox{strongly in $\LT 2 H$,} \label{eq:conv_phi_l2h_n} \\
            & \sigmaden \to \sigmade \quad \hbox{strongly in $\LT 2 H$,} \label{eq:conv_sigma_l2h_n} \\
            & \phiden - \dg \muden \to \phide - \dg \mude \quad \hbox{strongly in $\LT 2 V$.} \label{eq:conv_phidgmu_l2v_n}
        \end{align}
        Finally, from \eqref{eq:conv_phi_l2h_n}--\eqref{eq:conv_phidgmu_l2v_n}, up to a further non-relabelled subsequence, we also infer that
        \begin{align}
            & \phiden \to \phide \quad \hbox{a.e.~in $Q_T$,} \label{eq:conv_phi_ae_n} \\
            & \sigmaden \to \sigmade \quad \hbox{a.e.~in $Q_T$,} \label{eq:conv_sigma_ae_n} \\
            & \phiden - \dg \muden \to \phide - \dg \mude \quad \hbox{a.e.~in $Q_T$.} \label{eq:conv_phidgmu_ae_n}
        \end{align}
        With the convergences above, passing to the limit in all the linear terms in \eqref{eq:varform:phi_eps_n}--\eqref{eq:varform:sigma_eps_n} is a standard matter, therefore we only comment on how to pass to the limit in the remaining nonlinear terms.
        First, by \eqref{eq:conv_phi_l2h_n}, \eqref{eq:conv_phi_ae_n}, the Dominated Convergence Theorem and the continuity and boundedness of $b_\eps$, one easily sees that
        \begin{equation}
            \label{eq:conv_b_n}
            b_\eps(\phiden) \to b_\eps(\phide) \quad \hbox{strongly in $\LT 2 H$.}
        \end{equation}
        Next, we comment on the convergence of the term $\Pi^n(\Psipe'(\phiden))$.
        The continuity of $\Psipe'$ and $\Psipe''$, their growth conditions \eqref{ass:Psi_eps}, the convergences \eqref{eq:conv_phi_l2h_n} and \eqref{eq:conv_phi_ae_n}, and the Dominated Convergence Theorem imply that
        \[
            \Psipe'(\phiden) \to \Psipe'(\phide) \quad
            \text{and} \quad
            \Psipe''(\phiden) \to \Psipe''(\phide) \quad \hbox{strongly in $\LT 2 H$.}
        \]
        As a consequence, the bound \eqref{eq:bound_phi_l2v_n}, the convergence \eqref{eq:conv_phi_l2v_n} and the regularity of $\Psipe$ also give that
        \[
            \nabla \Psipe'(\phiden) = \Psipe''(\phiden) \nabla \phiden \weak
            \Psipe''(\phide) \nabla \phide = \nabla \Psipe'(\phide)
            \quad \hbox{weakly in $\LT 2 H$.}
        \]
        Hence, we conclude that
        \[
            \Psipe'(\phiden) \weak \Psipe'(\phide)
            \quad \hbox{weakly in $\LT 2 V$.}
        \]
        Now, we recall that, since the spaces $W_n$ are generated by eigenfunctions of $\Ncal$, the $H$-projection $\Pi^n$ onto $W_n$ coincides with the $V$-projection.
        Then, by standard properties of the projection operators, together with the density of $\cup_{n \in \N} W_n$ in $V$, we further deduce that
        \begin{equation}
            \label{eq:conv_proj_n}
            \Pi^n(\Psipe'(\phiden)) \weak \Psipe'(\phide)
            \quad \hbox{weakly in $\LT 2 V$.}
        \end{equation}
        The combination of \eqref{eq:conv_mu_l2v_n}, \eqref{eq:conv_b_n} and \eqref{eq:conv_proj_n}, together with the uniform bounds \eqref{eq:unifbounds_eps_n}, finally gives the convergence of the gradient term on the left-hand side of \eqref{eq:varform:phi_eps}.
        Indeed, by \eqref{eq:conv_mu_l2v_n} and \eqref{eq:conv_proj_n}, we have that $\nabla \muden + \nabla \Pi^n(\Psipe'(\phiden)) \weak \nabla \mude + \nabla \Psipe'(\phide)$ weakly in $\LT 2 H$.
        By combining this with \eqref{eq:conv_b_n}, we infer that
        \[
            b_\eps(\phiden) \left( \nabla \muden + \nabla \Pi^n(\Psipe'(\phiden)) \right) \weak b_\eps(\phide) \left( \nabla \mude + \nabla \Psipe'(\phide) \right) \quad  \hbox{weakly in $\Lqt 1$.}
        \]
        However, by Banach-Alaoglu's Theorem, the first uniform bound in \eqref{eq:unifbounds_eps_n}, together with the boundedness of $b_\eps$, implies the existence of some $\xi \in \LT 2 H$ such that
        \[
            b_\eps(\phiden) \left( \nabla \muden + \nabla \Pi^n(\Psipe'(\phiden)) \right) \weak \xi \quad \hbox{weakly in $\LT 2 H$.}
        \]
        Then, uniqueness of the weak limits gives that $\xi = b_\eps(\phide) \left( \nabla \mude + \nabla \Psipe'(\phide) \right)$, providing the desired convergence.
        Next, we deal with the convergence of the reaction term on the right-hand sides of \eqref{eq:varform:phi_eps_n} and \eqref{eq:varform:sigma_eps_n}.
        First, by \eqref{eq:conv_phi_l2h_n}, \eqref{eq:conv_phi_ae_n}, the Dominated Convergence Theorem and the continuity and boundedness of $P_\eps$, one easily sees that
        \begin{equation}
            \label{eq:conv_P_n}
            P_\eps(\phiden) \to P_\eps(\phide) \quad \hbox{strongly in $\LT 2 H$.}
        \end{equation}
        Second, \eqref{eq:conv_mu_l2v_n}, \eqref{eq:conv_sigma_l2v_n}, \eqref{eq:conv_phidgmu_l2v_n} and \eqref{eq:conv_proj_n} readily imply that
        \begin{equation}
            \label{eq:conv_R_n}
            R^n_{\delta, \eps}(\phiden, \muden, \sigmaden) \weak R_{\delta, \eps}(\phide, \mude, \sigmade) \quad \hbox{weakly in $\LT 2 H$,}
        \end{equation}
        where
        \[
            R_{\delta, \eps}(\phide, \mude, \sigmade) = \sigmade + \chi \left( 1 - \left( \phide - \dg \mude \right) \right) - (\mude + \Psipe'(\phide)).
        \]
        Hence, \eqref{eq:conv_P_n} and \eqref{eq:conv_R_n} yield
        \[
            P_\eps(\phiden) R^n_{\delta, \eps}(\phiden, \muden, \sigmaden) \weak P_\eps(\phide) R_{\delta,\eps}(\phide, \mude, \sigmade) \quad \hbox{weakly in $\Lqt 1$.}
        \]
        Then, arguing as above, the uniform boundedness of $P_\eps$, the last bound in \eqref{eq:unifbounds_eps_n} and Banach--Alaoglu's Theorem imply that
        \[
            P_\eps(\phiden) R^n_{\delta, \eps}(\phiden, \muden, \sigmaden) \weak P_\eps(\phide) R_{\delta,\eps}(\phide, \mude, \sigmade) \quad \hbox{weakly in $\LT 2 H$,}
        \]
        which results in the convergence of the full reaction term.
        Finally, the last nonlinear term to be discussed is the one involving $\Psi'_-$, but this is again a direct consequence of the continuity of $\Psi'_-$, its linear growth by \ref{ass:psi}, \eqref{eq:conv_phidgmu_l2v_n}, \eqref{eq:conv_phidgmu_ae_n} and the Dominated Convergence Theorem, yielding
        \begin{equation}
            \label{eq:conv_Psim_n}
            \Psi'_-\left( \phiden - \dg \muden \right) \to \Psi'_-\left( \phide - \dg \mude \right) \quad \hbox{strongly in $\LT 2 H$.}
        \end{equation}

        Lastly, we need to verify that the initial conditions $\phi_0 \in H$ and $\sigma_0 \in H$ are attained.
        However, this is a standard matter, because the regularities above and some standard embeddings imply that $\phi, \sigma \in \CT 0 H$.
    \end{proof}

    \section{Proof of Theorem \ref{thm:weaksols_delta}}
    \label{relaxprob}
    
    Here, we start from Proposition \ref{prop:weaksols_eps} and let $\eps$ go to $0$
    along a suitable sequence.
    The main ingredients are some uniform bounds, independent of $\eps$ combined with a new entropy estimate. The latter  exploits the structure depending on the degenerate mobility as well as on the single-well potential.
    We will make use of all the regularising properties introduced in \eqref{eq:reg_mob}--\eqref{ass:Peta_eps}.

        By Proposition \ref{prop:weaksols_eps}, for any $\eps > 0$, there exists a weak solution $(\phide, \mude, \sigmade)$ to the regularised problem \eqref{eq:phieps}--\eqref{eq:iceps} such that
        \begin{align*}
            & \phide \in \HT 1 {V^*} \cap \LT 2 V, \\
            & \mude \in \LT 2 V, \\
            & \phide - \dg \mude \in \HT 1 {V^*} \cap \LT \infty V \cap \LT 2 W, \\
            & \sigmade \in \HT 1 {V^*} \cap \LT \infty H \cap \LT 2 V,
        \end{align*}
        satisfying the identities \eqref{eq:varform:phi_eps}--\eqref{eq:varform:sigma_eps}.
        We now perform some \emph{a priori} estimates which are independent of $\eps > 0$.

        \textsc{Energy estimate.}
        For the main energy estimate, we argue as in the proof of Proposition \ref{prop:weaksols_eps}.
        Then, we observe that the energy inequality \eqref{eq:energyineq_eps_n} is preserved as $n \to + \infty$ thanks to the weak lower semicontinuity of the norms, the weak and strong convergences stated in \eqref{eq:conv_phi_l2v_n}--\eqref{eq:conv_R_n}, and in \eqref{eq:conv_initialenergy_n}.
        Thus, for almost any $t \in (0,T)$, we have
        \begin{equation}
            \label{eq:energyineq_eps}
            \begin{split}
            & \beta \int_\Omega \abs*{\phide(t) - \dg \mude(t)}^2 \de x
            + \mezzo \dg \int_\Omega \abs{\mude(t)}^2 \, \de x \\
            & \qquad
            + \frac{\gamma}{2} \int_\Omega \abs*{\nabla \left(\phide(t) - \dg \mude(t) \right)}^2 \,\de x
            + \alpha \int_\Omega \abs{\sigmade(t)}^2 \, \de x \\
            & \qquad
            + \int_0^t \int_\Omega b_\eps(\phide) \abs*{\nabla (\mude + \Psipe'(\phide))}^2 \, \de x \, \de s \\
            & \qquad
            + \int_0^t \int_\Omega \abs*{\nabla \left(\sigmade + \chi \left( 1 - \left(\phide - \dg \mude \right) \right) \right)}^2 \, \de x \, \de s
            \\
            & \qquad
            + \int_0^t \int_\Omega \left[ \sqrt{P_\eps(\phide)} \left( \sigmade + \chi \left( 1 - \left( \phide - \frac{\delta}{\gamma} \mude \right) \right) - (\mude + \Psipe'(\phide)) \right) \right]^2 \, \de x \, \de s \\
            & \quad \le
            \mathcal{E}_{\delta}(\phi_0, \sigma_0) + C \abs{\Omega},
            \end{split}
        \end{equation}
        where $C > 0$ is a constant independent of $\eps$, $\delta$,
        and $T$.
        Then, \eqref{eq:energyineq_eps} provides the following uniform bounds:
        \begin{equation}
            \label{eq:unifbounds_eps}
            \begin{split}
                & \norm*{\phide - \dg \mude}_{\LT \infty V} \le C, \\
                & \sqrt{\delta} \norm{\mude}_{\LT \infty H} \le C, \\
                & \norm{\sigmade}_{\LT \infty H} \le C, \\
                & \norm*{\sqrt{b_\eps(\phide)} \nabla (\mude + \Psipe'(\phide))}_{\LT 2 H} \le C, \\
                & \norm*{\nabla \left( \sigmade + \chi \left( 1 - \left(\phide - \dg \mude \right) \right) \right)}_{\LT 2 H} \le C, \\
                & \norm*{\sqrt{P_\eps(\phide)} \, R_{\delta, \eps}(\phide, \mude, \sigmade)}_{\LT 2 H}\le C,
            \end{split}
        \end{equation}
        with  $C > 0$ independent of $\eps$ and $T$, and depending on $\delta$ only through $\mathcal{E}_\delta(\phi_0, \sigma_0)$.
        From the first and fifth bounds in \eqref{eq:unifbounds_eps} it also follows that
        \begin{equation}
            \label{eq:bound_sigma_l2v_eps}
            \norm{\sigmade}^2_{\LT 2 V} \le C.
        \end{equation}

        \textsc{Estimates on the time derivatives.}
        Going back to the Galerkin discretisation used in the proof of Proposition \ref{prop:weaksols_eps}, one can also see that estimates \eqref{eq:bound_dtphi_n}, \eqref{eq:bound_dtsigma_n}, and \eqref{eq:bound_dtphimu_n} on the time derivatives hold with a constant independent of $n$ and $\eps$.
        Moreover, their derivation relies only on the uniform bounds given by the energy estimate.
        Then, using the convergences \eqref{eq:conv_phi_h1vs_n}--\eqref{eq:conv_phidgmu_h1vs_n}, and the weak lower semicontinuity of the norms, we get
        \begin{equation}
            \label{eq:bounds_dt_eps}
            \begin{split}
            & \norm{\partial_t \phide}_{\LT 2 {V^*}} \le C, \quad \norm{\partial_t \sigmade}_{\LT 2 {V^*}} \le C, \\
            & \norm*{\partial_t \left(\phide - \dg \mude \right)}_{\LT 2 {V^*}} \le C,
            \end{split}
        \end{equation}
        with $C > 0$ independent of $\eps$.

        \textsc{Entropy estimate.}
        Given the approximate entropy density function $\eta_\eps \in C^2(\R)$ defined by \eqref{eq:reg_entr}, for any $\eps > 0$, we define the approximate entropy functional as follows
    	\begin{equation}
    		\label{eq:entropy_eps}
    		\mathcal{S}_\eps(\phide) = \int_\Omega \eta_\eps(\phide) \, \de x.
    	\end{equation}
        We observe that the entropy functional $\mathcal{S}_\eps$ is bounded below uniformly in $\eps$, since $\eta_\eps$ is a $C^2$ convex function due to \ref{ass:b}, \ref{ass:entropy} and \eqref{eq:reg_entr}.
        Then, one formally deduces the entropy estimate by differentiating $\mathcal{S}_\eps(\phide)$ in time and substituting \eqref{eq:phieps}.
        To do this rigorously, we first recall that, by Proposition \ref{prop:weaksols_eps}, $\phide \in \LT 2 V$ for any $\eps > 0$.
        Consequently, since $\eta_\eps \in C^2(\R)$ and, by \eqref{eq:reg_entr_expl}, has a linearly bounded first derivative and a globally bounded second derivative, it follows that $v = \eta_\eps'(\phide) \in \LT 2 V$ is a valid test function for \eqref{eq:varform:phi_eps}.
        With this choice, we get
        \begin{equation}
            \label{eq:entropyest_eps}
            \begin{split}
            \ddt S_\eps(\phide)
            &
            = \duality{ \partial_t \phide \, \eta_\eps'(\phide) }_V \\
            & = - \int_\Omega b_\eps(\phide) \nabla (\mude + \Psipe'(\phide)) \cdot \eta''(\phide) \nabla \phide \, \de x \\
            & \quad + \int_\Omega P_\eps(\phide) \left(\sigmade + \chi \left( 1 - \left( \phide - \dg \mude \right) \right) - (\mude + \Psipe'(\phide)) \right) \eta_\eps'(\phide) \, \de x \\
            & := I_1 + I_2.
            \end{split}
        \end{equation}
        To estimate $I_1$, we argue as in \cite{PP2021}, that is, we use \eqref{eq:reg_entr} to cancel out the degenerate mobility and integrate by parts.
        More precisely, we have that
        \begin{align*}
            I_1 & =
            - \int_\Omega b_\eps(\phide) \nabla (\mude + \Psipe'(\phide)) \cdot \eta_\eps''(\phide) \nabla \phide \, \de x \\
            & = - \int_\Omega \left( \nabla \mude \cdot \nabla \phide + \nabla (\Psipe'(\phid)) \cdot \nabla \phide \right) \, \de x \\
            & = - \int_\Omega \nabla \mude \cdot \left( \nabla \left(\phide - \dg \mude \right) + \dg \nabla \mude \right)\, \de x
            - \int_\Omega \Psipe''(\phide) \abs{\nabla \phide}^2 \, \de x \\
            & = \int_\Omega \mude \, \Delta \left( \phide - \dg \mude \right) \, \de x - \dg \int_\Omega \abs{\nabla \mude}^2 \, \de x - \int_\Omega \Psipe''(\phide) \abs{\nabla \phide}^2 \, \de x,
        \end{align*}
        where the first integral on the last line is well-defined due to the fact that $\phide - \dg \mude \in \LT 2 {\Hx2}$ for any $\eps > 0$ by Proposition \ref{prop:weaksols_eps}.
        Then, we recall that \eqref{eq:mueps} can be rewritten in the following form (cf. \eqref{eq:mueps_prime}):
        \[
            \mude = - \gamma \Delta \left( \phide - \dg \mude \right) + \Psi'_- \left( \phide - \dg \mude \right) - \chi \sigmade,
        \]
        which now holds in strong form due to the regularities given by Proposition \ref{prop:weaksols_eps}.
        Hence, on account of this equation, we rewrite the first integral on the last line above as
        \begin{align*}
            & \int_\Omega \mude \, \Delta \left( \phide - \dg \mude \right) \, \de x \\
            & \quad = - \gamma \int_\Omega \abs*{\Delta \left( \phide - \dg \mude \right)}^2 \, \de x
            + \int_\Omega \Psi_-'\left(\phide - \dg \mude \right) \Delta \left(\phide - \dg \mude \right) \, \de x \\
            & \qquad - \chi \int_\Omega \sigmade \Delta \left(\phide - \dg \mude \right) \, \de x \\
            & \quad \le - \gamma \int_\Omega \abs*{\Delta \left( \phide - \dg \mude \right)}^2 \, \de x
            + \int_\Omega \Psi_-''\left(\phide - \dg \mude \right) \abs*{\nabla \left(\phide - \dg \mude \right)}^2 \, \de x \\
            & \qquad + \frac{\gamma}{2} \int_\Omega \abs*{\Delta \left( \phide - \dg \mude \right)}^2 \, \de x + \frac{\chi^2}{2 \gamma} \int_\Omega \abs{\sigmade}^2 \, \de x \\
            & \quad \le - \frac{\gamma}{2} \int_\Omega \abs*{\Delta \left( \phide - \dg \mude \right)}^2 \, \de x
            + \norm*{\Psi_-''}_{\infty} \int_\Omega \abs*{\nabla \left(\phide - \dg \mude \right)}^2 \, \de x
            + \frac{\chi^2}{2 \gamma} \int_\Omega \abs{\sigmade}^2 \, \de x,
        \end{align*}
        where we also used \ref{ass:psi}  and the Cauchy--Schwarz and the Young inequalities to handle the chemotactic term.
        For $I_2$, instead, we proceed by employing the Cauchy--Schwarz and the Young inequalities as follows:
        \begin{align*}
            I_2 & =
            \int_\Omega P_\eps(\phide) R_{\delta,\eps}(\phide, \mude, \sigmade) \eta_\eps'(\phide) \, \de x \\
            & = \int_\Omega \sqrt{P_\eps(\phide)} \eta_\eps'(\phide) \sqrt{P_\eps(\phide)} R_{\delta,\eps}(\phide, \mude, \sigmade) \, \de x \\
            & \le \mezzo \norm*{\sqrt{P_\eps(\phide)} \eta_\eps'(\phide)}^2_{H}
            + \mezzo \norm*{\sqrt{P_\eps(\phide)} R_{\delta,\eps}(\phide, \mude, \sigmade)}^2_{H}.
        \end{align*}
        Now, the second term is bounded by \eqref{eq:unifbounds_eps}. Concerning the first one, we use hypothesis \eqref{ass:Peta_eps}, which holds uniformly in $\eps > 0$.
        Then, we obtain
        \begin{align*}
            \norm*{\sqrt{P_\eps(\phide)} \eta_\eps'(\phide)}^2_H
            & = \int_\Omega \abs*{\sqrt{P_\eps(\phide)} \eta_\eps'(\phide)}^2 \, \de x \\
            & \le 2c_4^2 \int_\Omega \abs{\phide}^2 \, \de x + 2c_5^2 \abs{\Omega} \\
            & \le 4c_4^2 \int_\Omega \abs*{\phide - \dg \mude}^2 \, \de x + 4c_4^2 \frac{\delta^2}{\gamma^2} \int_\Omega \abs{\mude}^2 \, \de x + 2c_5^2 \abs{\Omega} \\
            & \le 4c_4^2 \int_\Omega \abs*{\phide - \dg \mude}^2 \, \de x + \frac{4c_4^2 \delta_0}{\gamma} \frac{\delta}{\gamma} \int_\Omega \abs{\mude}^2 \, \de x + 2c_5^2 \abs{\Omega},
        \end{align*}
        for $\delta \in (0, \delta_0)$.
        Thus, by integrating \eqref{eq:entropyest_eps} in $(0,t)$, for any given $t \in (0,T)$, and incorporating the above computations, we infer that (see also \eqref{eq:energyineq_eps})
        \begin{align*}
            & \mathcal{S}_\eps(\phide(t))
            + \frac{\gamma}{2} \int_0^t \int_\Omega \abs*{\Delta \left( \phide - \dg \mude \right)}^2 \, \de x \, \de s \\
            & \qquad + \int_0^t \int_\Omega \Psipe''(\phide) \abs{\nabla \phide}^2 \, \de x \, \de s
            + \dg \int_0^t \int_\Omega \abs{\nabla \mude}^2 \, \de x \, \de s \\
            & \quad
            \le \mathcal{S}_\eps(\phi_0)
            + \norm*{\Psi_-''}_{\infty} \int_0^T \int_\Omega \abs*{\nabla \left(\phide - \dg \mude \right)}^2 \, \de x \, \de t
            + \frac{\chi^2}{2 \gamma} \int_0^T \int_\Omega \abs{\sigmade}^2 \, \de x \, \de t \\
            & \qquad
            + \int_0^T \int_\Omega \left[ \sqrt{P_\eps(\phide)} \left( \sigmade + \chi \left( 1 - \left( \phide - \frac{\delta}{\gamma} \mude \right) \right) - (\mude + \Psipe'(\phide)) \right) \right]^2 \, \de x \, \de t \\
            & \qquad
            + 4c_4^2 \int_0^T \int_\Omega \abs*{\phide - \dg \mude}^2 \, \de x \, \de t
            + \frac{4c_4^2 \delta_0}{\gamma} \frac{\delta}{\gamma} \int_0^T \int_\Omega \abs{\mude}^2 \, \de x \, \de t
            + 2c_5^2 \abs{\Omega} T \\
            & \quad \le \mathcal{S}_\eps(\phi_0) + C_1 \mathcal{E}_{\delta}(\phi_0, \sigma_0) + C_2,
        \end{align*}
        where the constants $C_1, C_2 > 0$ only depend on the parameters of the system, $\abs{\Omega}$ and $T$.
        Moreover, by \eqref{ass:iniz_eps_eta}, one easily sees that $\mathcal{S}_\eps(\phi_0) \le \mathcal{S}(\phi_0)$, since $\phi_0$ satisfies \ref{ass:iniz}.
        Hence, we have the entropy inequality
        \begin{equation}
            \label{eq:entropyineq_eps}
            \begin{split}
            & \mathcal{S}_\eps(\phide(t))
            + \frac{\gamma}{2} \int_0^t \int_\Omega \abs*{\Delta \left( \phide - \dg \mude \right)}^2 \, \de x \, \de s \\
            & \qquad + \int_0^t \int_\Omega \Psipe''(\phide) \abs{\nabla \phide}^2 \, \de x \, \de s
            + \dg \int_0^t \int_\Omega \abs{\nabla \mude}^2 \, \de x \, \de s \\
            & \quad \le \mathcal{S}(\phi_0) + C_1 \mathcal{E}_{\delta}(\phi_0, \sigma_0) + C_2,
            \end{split}
        \end{equation}
        which in turn, also due to the uniform bound from below on $\mathcal{S}_\eps$, provides the following uniform bounds:
        \begin{equation}
    \label{eq:unifbounds_entropy_eps}
            \norm*{\Delta \left( \phide - \dg \mude \right)}_{\LT 2 H} \le C, \qquad
                 \sqrt{\delta} \norm{\nabla \mude}_{\LT 2 H} \le C.
        \end{equation}
        By elliptic regularity theory, the first bound in \eqref{eq:unifbounds_entropy_eps} and the boundary condition \eqref{eq:bceps}, together with the first bound in \eqref{eq:unifbounds_eps}, imply that
        \begin{equation}
       \label{eq:bound_phidgmu_h2_eps}
            \norm*{\phide - \dg \mude}_{\LT 2 {\Hx 2}} \le C.
        \end{equation}
        Additionally, combining the second bound in \eqref{eq:unifbounds_eps} with the second bound in \eqref{eq:unifbounds_entropy_eps}, we get that
        \begin{equation}
            \label{eq:bound_mu_l2v_eps}
            \sqrt{\delta} \norm{\mude}_{\LT 2 V} \le C.
        \end{equation}
        Then, the first bound in \eqref{eq:unifbounds_eps} and \eqref{eq:bound_mu_l2v_eps}  imply that
        \begin{equation}
            \label{eq:bound_phi_l2v_eps}
            \norm{\phide}_{\LT 2 V} \le C,
        \end{equation}
        with a constant $C > 0$ independent of  $\eps$.

        \textsc{Passage to the limit.}
        We now fix $\delta > 0$ and pass to the limit in the regularisation along a suitable vanishing subsequence $\{\eps_n\}_{n\in \mathbb{N}}$ to recover a weak solution to the relaxed system (see \eqref{eq:varform:phi}--\eqref{eq:varform:sigma}).
        Recalling the end of the proof of Proposition \ref{prop:weaksols_eps}, we argue in a similar way.
        If not further specified, all the convergences below have to be intended along a suitable vanishing subsequence of $\eps$.
        The uniform bounds \eqref{eq:bound_sigma_l2v_eps} and  \eqref{eq:bound_phidgmu_h2_eps}--\eqref{eq:bound_phi_l2v_eps} imply, up to a non-relabelled subsequence, the weak convergences
        \begin{align}
            & \phide \weak \phid \quad \hbox{weakly in $\LT 2 V$,} \label{eq:conv_phi_l2v_eps} \\
            & \mude \weak \mud \quad \hbox{weakly in $\LT 2 V$,} \label{eq:conv_mu_l2v_eps} \\
            & \sigmade \weakstar \sigmad \quad \hbox{weakly star in $\LT \infty H \cap \LT 2 V$,} \label{eq:conv_sigma_l2v_eps} \\
            & \phide - \dg \mude \weakstar \phid - \dg \mud \quad \hbox{weakly star in $\LT \infty V \cap \LT 2 {\Hx2}$.} \label{eq:conv_phidgmu_l2w_eps}
        \end{align}
        Similarly, by \eqref{eq:bounds_dt_eps}, we deduce
        \begin{align}
            & \phide \weak \phid \quad \hbox{weakly in $\HT 1 {V^*}$,} \label{eq:conv_phi_h1vs_eps} \\
            & \sigmade \weak \sigmad \quad \hbox{weakly in $\HT 1 {V^*}$,} \label{eq:conv_sigma_h1vs_eps} \\
            & \phide - \dg \mude \weak \phid - \dg \mud \quad \hbox{weakly in $\HT 1 {V^*}$.} \label{eq:conv_phidgmu_h1vs_eps}
        \end{align}
        Then, on account of known compact embeddings  (see \cite[Section 8, Corollary 4]{S1986}), \eqref{eq:conv_phi_l2v_eps}--\eqref{eq:conv_phidgmu_l2w_eps} and \eqref{eq:conv_phi_h1vs_eps}--\eqref{eq:conv_phidgmu_h1vs_eps} entail the strong convergences:
        \begin{align}
            & \phide \to \phid \quad \hbox{strongly in $\LT 2 H$,} \label{eq:conv_phi_l2h_eps} \\
            & \sigmade \to \sigmad \quad \hbox{strongly in $\LT 2 H$,} \label{eq:conv_sigma_l2h_eps} \\
            & \phide - \dg \mude \to \phide - \dg \mude \quad \hbox{strongly in $\LT 2 V$.} \label{eq:conv_phidgmu_l2v_eps}
        \end{align}
        Finally, from \eqref{eq:conv_phi_l2h_eps}--\eqref{eq:conv_phidgmu_l2v_eps}, up to a further non-relabelled subsequence, we also infer that
        \begin{align}
            & \phide \to \phid \quad \hbox{a.e.~in $Q_T$,} \label{eq:conv_phi_ae_eps} \\
            & \sigmade \to \sigmad \quad \hbox{a.e.~in $Q_T$,} \label{eq:conv_sigma_ae_eps} \\
            & \phide - \dg \mude \to \phid - \dg \mud \quad \hbox{a.e.~in $Q_T$.} \label{eq:conv_phidgmu_ae_eps}
        \end{align}
        Then, the limit variables $(\phid, \mud, \sigmad)$ satisfy the regularity properties stated in Theorem \ref{thm:weaksols_delta}.

        Next, we aim to prove that $\phid$ takes its values in $[0,1]$ almost everywhere in $Q_T$.
        To do this, we follow the approach used in \cite{AACG2017} and \cite{PP2021} for the Cahn--Hilliard equation with single-well potential and degenerate mobility.
        First, we look at the regularised entropy density $\eta_\eps$ and observe that, by \eqref{eq:reg_mob} and \eqref{eq:reg_entr}--\eqref{eq:reg_entr_expl}, for $s \ge 1$, it holds
        \[
            \eta_\eps(s) = \eta(1-\eps) + \eta'(1-\eps)(s - (1-\eps)) + \mezzo \frac{1}{b(1-\eps)} (s - (1-\eps))^2 \ge \frac{1}{2 b(1-\eps)} (s - 1)^2,
        \]
        since $\eta(1-\eps)$ and $\eta'(1-\eps)$ are positive by definition.
        Then, observe that
        \begin{align*}
            \int_\Omega (\phide(t) - 1)_+^2 \, \de x
            & = \int_{\{\phide(t) > 1\}} (\phide(t) - 1)^2 \, \de x \\
            & \le 2 b(1-\eps) \int_{\{\phide(t) > 1\}} \eta_\eps(\phide(t)) \, \de x \\
            & \le 2 b(1 - \eps) \int_\Omega \eta_\eps(\phide(t)) \, \de x \\
            & = 2 b(1 - \eps) \mathcal{S}_\eps(\phide(t))
            \le C b(1-\eps),
        \end{align*}
        for almost any $t \in (0,T)$, where $C>0$ is a constant independent of $\eps$ (see \eqref{eq:entropyineq_eps}).
        Then, recalling the strong convergence \eqref{eq:conv_phi_l2h_eps} and that $b(1-\eps) \to b(1)=0$ as $\eps \to 0$ by \ref{ass:b}, passing to the limit as $\eps \to 0$ we deduce that
        \[
            \int_\Omega (\phid(t) - 1)^2_+ \, \de x = 0 \quad \hbox{for a.e.~$t \in (0,T)$,}
        \]
        which implies that $\phid \le 1$ almost everywhere in $Q_T$.
        A similar reasoning on $\eta_\eps$ for $s \le 0$, relying again on the entropy inequality \eqref{eq:entropyineq_eps}, also gives $\phid \ge 0$ almost everywhere in $Q_T$.

        Finally, it remains to show that the limit functions $(\phid, \sigmad, \mud)$ satisfy the identities \eqref{eq:varform:phi}--\eqref{eq:varform:sigma}.
        As in the proof of Proposition \ref{prop:weaksols_eps}, the weak convergences above easily imply the convergence of all the linear terms in \eqref{eq:varform:phi_eps}--\eqref{eq:varform:sigma_eps} as $\eps \to 0$.
        Then, we just comment on how to pass to the limit in the nonlinear terms.
        First, we observe that, by its definition \eqref{eq:reg_mob}, $b_\eps \to b$ uniformly in $[0,1]$.
        Hence, together with \eqref{eq:conv_phi_ae_eps} and the fact that $0 \le \phid \le 1$, this implies that $b_\eps(\phide) \to b(\phid)$ almost everywhere in $Q_T$.
        Then, the boundedness of $b$ (see \ref{ass:b}) and the Dominated Convergence Theorem yield
        \begin{equation}
        \label{eq:conv_b_eps}
            b_\eps(\phide) \to b(\phid) \quad \hbox{strongly in $\LT 2 H$.}
        \end{equation}
        By a completely analogous argument, \ref{ass:P}, \eqref{eq:reg_prol}, \eqref{eq:conv_phi_ae_eps} and the Dominated Convergence Theorem imply that
        \begin{equation}
        \label{eq:conv_P_eps}
            P_\eps(\phide) \to P(\phid) \quad \hbox{strongly in $\LT 2 H$.}
        \end{equation}
        We now consider the gradient term on the left-hand side of \eqref{eq:varform:phi_eps}, which, for any $\eps > 0$, can be rewritten as $b_\eps(\phide) \nabla \mude + b_\eps(\phide) \Psipe''(\phide) \nabla \phide$.
        Our aim is to show that this converges weakly in $\LT 2 H$.
        For the first term, we have (see \eqref{eq:bound_mu_l2v_eps})
        \begin{equation}
        \label{eq:bound_bmu_eps}
            \norm{b_\eps(\phide) \nabla \mude}_{\LT 2 H} \le C_\delta.
        \end{equation}
        Then, arguing as in the last part of the proof of Proposition \ref{prop:weaksols_eps}, the bound  \eqref{eq:bound_bmu_eps}  and the weak and strong convergences \eqref{eq:conv_mu_l2v_eps} and \eqref{eq:conv_b_eps} imply that
        \begin{equation}
        \label{eq:conv_Jeps1}
            b_\eps(\phide) \nabla \mude \weak b(\phid) \nabla \mud \quad \hbox{weakly in $\LT 2 H$.}
        \end{equation}
        Next, we observe that, since $b \Psi_+'' \in C^0([0,1])$ by \ref{ass:b}, $b_\eps \Psipe'' \to b \Psi_+''$ uniformly in $[0,1]$ as $\eps \to 0$.
        Hence, together with \eqref{eq:conv_phi_ae_eps} and the fact that $0 \le \phid \le 1$, this implies that $b_\eps(\phide) \Psipe''(\phide) \to b(\phid) \Psi_+''(\phid)$ almost everywhere in $Q_T$.
        Then, thanks to the boundedness of $b_\eps \Psipe''$ (see \eqref{eq:reg_mob} and \eqref{eq:reg_pot_second}), the Dominated Convergence Theorem yields
        \begin{equation}
        \label{eq:conv_bpsi_eps}
            b_\eps(\phide) \Psipe''(\phide) \to b(\phid) \Psi_+''(\phid) \quad \hbox{strongly in $\LT 2 H$.}
        \end{equation}
        Finally, to deduce the desired convergence, we combine \eqref{eq:unifbounds_eps} and \eqref{eq:bound_bmu_eps} to infer that
        \[
            \norm{b_\eps(\phide) \Psipe''(\phide) \nabla \phide}_{\LT 2 H} \le C_\delta,
        \]
        which, together with \eqref{eq:conv_phi_l2v_eps} and \eqref{eq:conv_bpsi_eps}, implies that
        \begin{equation}
        \label{eq:conv_Jeps2}
            b_\eps(\phide) \Psipe''(\phide) \nabla \phide \weak b(\phid) \Psi_+''(\phid) \nabla \phid \quad  \hbox{weakly in $\LT 2 H$.}
        \end{equation}
        We now deal with the reaction term $P_\eps(\phide) R_{\delta, \eps}(\phide, \mude, \sigmade)$.
        Due to the singularity of $\Psi_+'$, we have to split $R_{\delta, \eps}(\phide, \mude, \sigmade)$ into two parts and treat them separately, namely,
        \[
            R_{\delta, \eps}(\phide, \mude, \sigmade) = \left( \sigmade + \chi \left( 1 - \left(\phide - \dg \mude \right) \right) - \mude \right) - \Psipe'(\phide).
        \]
        Then, we observe that, by \eqref{eq:conv_mu_l2v_eps}--\eqref{eq:conv_phidgmu_l2w_eps} and the Sobolev embedding $V \hookrightarrow \Lx 6$, it holds
        \[
            \sigmade + \chi \left( 1 - \left(\phide - \dg \mude \right) \right) - \mude \weak \sigmad + \chi \left( 1 - \left(\phid - \dg \mud \right) \right) - \mud
        \]
        weakly in $\LT 2 {\Lx 6}$.
        Combining this with the strong convergence \eqref{eq:conv_P_eps}, we easily see that, for any $v \in V \hookrightarrow \Lx 3$,
        \begin{align*}
            & \int_\Omega P_\eps(\phide) \left( \sigmade + \chi \left( 1 - \left(\phide - \dg \mude \right) \right) - \mude \right) v \, \de x \\
            & \quad \to \int_\Omega P(\phid) \left( \sigmad + \chi \left( 1 - \left(\phid - \dg \mud \right) \right) - \mud \right) v \, \de x,
        \end{align*}
        almost everywhere in $[0,T]$ as $\eps \to 0$.
        For the second term, instead, we observe that, since $P \Psi_+' \in C^0([0,1])$ by \ref{ass:P}, $P_\eps \Psipe' \to P \Psi_+'$ uniformly in $[0,1]$ as $\eps \to 0$.
        Hence, together with \eqref{eq:conv_phi_ae_eps} and the fact that $0 \le \phid \le 1$, this implies that $P_\eps(\phide) \Psipe'(\phide) \to P(\phid) \Psi_+'(\phid)$ almost everywhere in $Q_T$.
        To apply the Dominated Convergence Theorem, we now need a uniform bound on $P_\eps(\phide) \Psipe'(\phide)$ in some $L^p$ space.
        By looking at their approximations \eqref{eq:reg_pot} and \eqref{eq:reg_prol}, and relying on the fact that $P \Psi_+' \in C^0([0,1])$ and $P(s) \le c_3^2 b^2(s)$ by \ref{ass:P}, one can easily infer that, for any $0 < \eps < \eps_0$,
        \begin{equation}
            \label{eq:growth_PPsi_eps}
            \abs{P_\eps(s) \Psipe'(s)} \le C + C \abs{s} \quad \hbox{for any $s \in \R$,}
        \end{equation}
        for some $C > 0$ independent of $\eps$.
        Indeed, this is trivial for $\eps_0 \le s \le 1 - \eps_0$. For $s \ge 1 - \eps_0$, instead, one can proceed as follows:
        \begin{align*}
            \abs{P_\eps(s) \Psipe'(s)}
            & = \abs{P(1-\eps) \Psi_+'(1-\eps) + P(1-\eps) \Psi_+''(1-\eps) (s - (1-\eps))} \\
            & \le \norm{P \Psi'}_{C^0([0,1])} + c_3^2 \abs{b(1-\eps)} \abs{b(1-\eps) \Psi_+''(1-\eps)} \abs{s - (1-\eps)} \\
            & \le \norm{P \Psi'}_{C^0([0,1])} + c_3^2 \norm{b}_{C^0([0,1])} \norm{b \Psi_+''}_{C^0([0,1])}  \abs{s}
            \le C + C \abs{s}.
        \end{align*}
        A similar argument also provides the bound for $s \le \eps_0$.
        Then, \eqref{eq:growth_PPsi_eps} and the strong convergence \eqref{eq:conv_phi_l2h_eps} are enough to apply the generalised Dominated Convergence Theorem to conclude that
        \[
            P_\eps(\phide) \Psipe'(\phide) \to P(\phid) \Psi_+'(\phid) \quad \hbox{strongly in $\LT 2 H$.}
        \]
        Thus, we can pass to the limit in the reaction term.
        Finally, we are left to handle the nonlinear term $\Psi_-'(\phide - \dg \mude)$, but this is immediate thanks to the linear growth of $\Psi_-'$ and the strong convergence \eqref{eq:conv_phidgmu_l2v_eps} (cf. Proposition \ref{prop:weaksols_eps}).

        Summing up, we have shown that the limit variables $(\phid, \mud, \sigmad)$ satisfy the identities \eqref{eq:varform:phi}--\eqref{eq:varform:sigma}.
        We conclude the proof of Theorem \ref{thm:weaksols_delta} by stressing that the initial conditions $\phid(0) = \phi_0$ and $\sigmad(0) = \sigma_0$ are trivially satisfied in $H$, due to the standard embedding $\HT 1 {V^*} \cap \LT 2 {V} \hookrightarrow \CT 0 H$.

    \section{Proof of Theorem  \ref{thm:conv_deltato0}}
    \label{limitprob}

        By Theorem \ref{thm:weaksols_delta}, for any $\delta > 0$, there exists  $(\phid, \sigmad, \mud)$ that satisfies \eqref{eq:varform:phi}--\eqref{eq:varform:sigma} and 
        \begin{align*}
            & \phid \in \HT 1 {V^*} \cap \LT 2 V, \\
            & \mud \in \LT 2 V, \\
            & \phid - \dg \mud \in \HT 1 {V^*} \cap \LT \infty V \cap \LT 2 W, \\
            & \sigmad \in \HT 1 {V^*} \cap \LT \infty H \cap \LT 2 V, \\
            & 0 \le \phid \le 1 \quad \hbox{a.e.~in $Q_T$.}
        \end{align*}
        We now proceed as in Theorem \ref{thm:weaksols_delta} and perform some \emph{a priori} estimates uniformly with respect to the relaxation parameter $\delta$.

        \textsc{Energy estimate.}
        We start from the energy inequality \eqref{eq:energyineq_eps} obtained in the proof of Theorem \ref{thm:weaksols_delta} for $\eps > 0$.
        Then, we observe that we can let $\eps$ go to $0$ along a suitable subsequence and, using weak lower semicontinuity of the norms along with the weak and strong convergences \eqref{eq:conv_phi_l2v_eps}--\eqref{eq:conv_phidgmu_ae_eps} and their implications. This gives, for any $t \in (0,T)$, the following energy inequality
        \begin{equation}
            \label{eq:energyineq_delta}
            \begin{split}
            & \beta \int_\Omega \abs*{\phid(t) - \dg \mud(t)}^2 \de x
            + \mezzo \dg \int_\Omega \abs{\mud(t)}^2 \, \de x \\
            & \qquad
            + \frac{\gamma}{2} \int_\Omega \abs*{\nabla \left(\phid(t) - \dg \mud(t) \right)}^2 \,\de x
            + \alpha \int_\Omega \abs{\sigmad(t)}^2 \, \de x \\
            & \qquad
            + \int_0^t \int_\Omega \abs*{\nabla \left(\sigmad + \chi \left( 1 - \left(\phid - \dg \mud \right) \right) \right)}^2 \, \de x \, \de s
            \\
            & \qquad
            + \int_0^t \int_\Omega \left[ \sqrt{P(\phid)} \left( \sigmad + \chi \left( 1 - \left( \phid - \frac{\delta}{\gamma} \mud \right) \right) - (\mud + \Psi_+'(\phid)) \right) \right]^2 \, \de x \, \de s \\
            & \quad \le
            \mathcal{E}_{\delta}(\phi_0, \sigma_0) + C \abs{\Omega},
            \end{split}
        \end{equation}
        where $C > 0$ is a constant independent of $\delta$ and $T$.
        We observe that the singularity of $\Psi''_+$ prevents us from sending to the limit as $\eps \to 0$ the term $\int_0^t \int_\Omega b_\eps(\phide) \abs*{\nabla \mude + \Psipe''(\phide) \nabla \phide}^2$ in \eqref{eq:energyineq_eps}, unless we assume a higher degeneracy for $b$, namely that $\sqrt{b} \Psi''_+ \in C^0$, but this would exclude the typical choice for $b$.
        However, this is not an issue since this term is well-defined and non-negative for any $\eps > 0$, so we can still pass to the limit in \eqref{eq:energyineq_eps} by neglecting it.
        At the same time, we observe that the other singular term involving $\sqrt{P_\eps(\phide)} \Psipe'(\phide)$ easily passes to the limit due to the second condition in \ref{ass:P}.
        Let us now show that the bound \eqref{eq:energyineq_delta} holds uniformly with respect to $\delta$. More precisely, we prove that
        \begin{equation}
            \label{eq:conv_initenergy_delta}
            \mathcal{E}_\delta(\phi_0, \sigma_0) \to \mathcal{E}(\phi_0, \sigma_0) \quad \hbox{as $\delta \to 0$,}
        \end{equation}
        where
        \[
            \mathcal{E}(\phi_0, \delta_0) = \int_\Omega \left( \Psi(\phi_0) + \frac{\gamma}{2} \abs{\nabla \phi_0}^2 + \mezzo \abs{\sigma_0}^2 + \chi \sigma_0 (1 - \phi_0) \right) \, \de x.
        \]
        Indeed, we recall that $\phid - \dg \mud\in C^0([0,T];V)$ and, by Remark \ref{rmk:initial_data}, given $\phi_0, \sigma_0 \in H$, $\phi_0 - \dg \mu(0) \in V$ can be obtained as the unique weak solution to the elliptic problem deduced by \eqref{eq:mu} at time $t = 0$:
        \begin{align*}
            & - \delta \Delta \left(\phi_0 - \dg \mud(0) \right)
            + \left( \phi_0 - \dg \mud(0) \right)
            = \phi_0
            - \dg \Psi'_- \left(\phi_0 - \dg \mud(0)\right) + \dg \chi \sigma_0 \quad && \hbox{ in } \Omega,\\
            & \partial_{\n} \left(\phi_0 - \dg \mud(0) \right) = 0 \quad && \hbox{ on } \partial\Omega.
        \end{align*}
        Now, since $\phi_0 \in V$ by \ref{ass:iniz2}, $\sigma_0 \in H$ and $\Psi_-'$ is of class $C^1$,  linearly bounded and with bounded derivative, one can show that (see Lemma \ref{lemma:singpert})
        \begin{equation}
        \label{singpert}
            \phi_0 - \dg \mud(0) \to \phi_0 \quad \hbox{strongly in $V$ as $\delta \to 0$.}
        \end{equation}
        Then, the convergence above and the continuity of $\Psi$ are enough to pass to the limit in the energy and conclude \eqref{eq:conv_initenergy_delta}.
        Hence, \eqref{eq:energyineq_delta} and \eqref{eq:conv_initenergy_delta} provide the following uniform estimates:
        \begin{equation}
            \label{eq:unifbounds_delta}
            \begin{split}
                & \norm*{\phid - \dg \mud}_{\LT \infty V} \le C, \quad \sqrt{\delta} \norm{\mud}_{\LT \infty H} \le C,
                \quad \norm{\sigmad}_{\LT \infty H} \le C, \\
                & \norm*{\nabla \left( \sigmad + \chi \left( 1 - \left(\phid - \dg \mud \right) \right) \right)}_{\LT 2 H} \le C, \\
                & \norm*{\sqrt{P(\phid)} \, R_{\delta}(\phid, \mud, \sigmad)}_{\LT 2 H} \le C,
            \end{split}
        \end{equation}
        with  $C > 0$ independent of $T$ and of $\delta$.
        From the first and fifth bound in \eqref{eq:unifbounds_delta} it also immediately follows that
        \begin{equation}
    \label{eq:bound_sigma_l2v_delta}
            \norm{\sigmad}_{\LT 2 V} \le C.
        \end{equation}

        \textsc{Estimates on the time derivatives.}
        We start from the uniform bounds \eqref{eq:bounds_dt_eps}.
        We recall that such bounds were obtained by means of the estimates \eqref{eq:bound_dtphi_n}, \eqref{eq:bound_dtsigma_n} and \eqref{eq:bound_dtphimu_n} derived within the Galerkin discretisation in the proof of Proposition \ref{prop:weaksols_eps}.
        Moreover, we stress that, for any $\eps >0$, their derivation only relied on the uniform bounds obtained through the main energy estimate at the approximated level, which, in view of \eqref{eq:conv_initenergy_delta}, now holds with a constant independent of $\eps$ and $\delta$.
        Therefore, by using the convergences \eqref{eq:conv_phi_h1vs_eps}--\eqref{eq:conv_phidgmu_h1vs_eps}, one obtains by the weak lower semicontinuity of the norms, that
        \begin{equation}
            \label{eq:bounds_dt_delta}
            \norm{\partial_t \phid}_{\LT 2 {V^*}} \le C, \quad
            \norm{\partial_t \sigmad}_{\LT 2 {V^*}} \le C, \quad
            \norm*{\partial_t \left(\phid - \dg \mud \right)}_{\LT 2 {V^*}} \le C,
        \end{equation}
        with $C > 0$ independent of $\delta$.

        \textsc{Entropy estimate.}
        Also in this case we start from the entropy inequality \eqref{eq:entropyineq_eps} derived in the proof of Theorem \ref{thm:weaksols_delta} for any $\eps > 0$.
        In view of \eqref{eq:conv_initenergy_delta}, we observe that the right-hand side of \eqref{eq:entropyineq_eps} is uniformly bounded in $\eps$ and $\delta$.
        Therefore, thanks to the convergences \eqref{eq:conv_mu_l2v_eps} and \eqref{eq:conv_phidgmu_l2w_eps}, by the weak lower semicontinuity of the norms we deduce that
        \begin{equation}
            \label{eq:unifbounds_entropy_delta}
            \norm*{\Delta \left( \phid - \dg \mud \right)}_{\LT 2 H} \le C, \quad
               \sqrt{\delta} \norm{\nabla \mud}_{\LT 2 H} \le C,
        \end{equation}
        with $C > 0$ independent of $\delta$.
        We remark that the singularity of $\Psi_+$ does not allow us take the limit of the full entropy inequality \eqref{eq:entropyineq_eps}. However, the uniform bound from below on $\mathcal{S}_\eps$ and the non-negativity of the third term, for any fixed $\eps > 0$, are enough to get \eqref{eq:unifbounds_entropy_delta}.
        Then, by elliptic regularity theory, the first bound in \eqref{eq:unifbounds_delta} together with the first bound in \eqref{eq:unifbounds_entropy_delta}, imply that
        \begin{equation}
            \label{eq:bound_phidgmu_h2_delta}
            \norm*{\phid - \dg \mud}_{\LT 2 {\Hx 2}} \le C.
        \end{equation}
        Additionally, combining the second bound in \eqref{eq:unifbounds_delta} with the second bound in \eqref{eq:unifbounds_entropy_delta}, we get that
        \begin{equation}
            \label{eq:bound_mu_l2v_delta}
            \sqrt{\delta} \norm{\mud}_{\LT 2 V} \le C.
        \end{equation}
        Then, the first bounds in \eqref{eq:unifbounds_delta} and \eqref{eq:bound_mu_l2v_delta} give
        \begin{equation}
            \label{eq:bound_phi_l2v_delta}
            \norm{\phid}_{\LT 2 V} \le C,
        \end{equation}
        with a constant $C > 0$ independent of $\delta$.

        \textsc{Passage to the limit.}
        Before passing to the limit as $\delta \to 0$, we rewrite the relaxed system \eqref{eq:phi}--\eqref{eq:ic} using the identity (see \eqref{eq:mu})
        \begin{equation}
		\label{eq:muprime}
		   \mud = - \gamma \Delta \left( \phid - \dg \mud \right)
		  + \Psi'_-\left(\phid - \dg \mud\right) - \chi \sigmad.
	    \end{equation}
        Thus, we obtain a formulation similar to \eqref{eq:phi00}--\eqref{eq:ic00}. More precisely, we get
        \begin{alignat}{2}
    		& \partial_t \phid
    		- \div \left( b(\phid) \nabla \left(- \gamma \Delta \left( \phid - \dg \mud \right) + \Psi_+'(\phid) + \Psi_-'\left( \phid - \dg \mud \right) - \chi \sigmad \right) \right) \notag \\
    		& \quad = P(\phid) \bigg( \sigmad + \chi \left( 1 - \left( \phid - \dg \mud \right) \right) + \gamma \Delta \left( \phid - \dg \mud \right) \notag \\
            & \qquad - \Psi_+'(\phid) - \Psi_-'\left( \phid - \dg \mud \right) + \chi \sigmad \bigg)
            \quad && \hbox{in $Q_T$,} \label{eq:phidelta00} \\
    		& \partial_t \sigmad -
    		\Delta \left( \sigmad + \chi \left( 1 - \left(\phid - \frac{\delta}{\gamma} \mud \right) \right) \right) \notag \\
    		& \quad = - P(\phid) \bigg( \sigmad + \chi \left( 1 - \left( \phid - \dg \mud \right) \right) + \gamma \Delta \left( \phid - \dg \mud \right) \notag \\
            & \qquad - \Psi_+'(\phid) - \Psi_-'\left( \phid - \dg \mud \right) + \chi \sigmad \bigg)
            \quad && \hbox{in $Q_T$,} \label{eq:sigmadelta00} \\
    		& \partial_{\n} \left( \phid - \frac{\delta}{\gamma} \mud \right) = \partial_{\n} \sigmad = 0, \notag \\
    		& \quad b(\phid) \partial_{\n} \left(- \gamma \Delta \left( \phid - \dg \mud \right) + \Psi_+'(\phid) + \Psi_-'\left( \phid - \dg \mud \right) - \chi \sigmad \right) = 0
    		\quad && \hbox{on $\Sigma_T$,} \label{eq:bcdelta00} \\
    		& \phid(0) = \phi_0,
    		\quad \sigmad(0) = \sigma_0
    		\quad && \hbox{in $\Omega$,} \label{eq:icdelta00}
	    \end{alignat}
       which has to be interpreted in the weak sense. 
       On account of Theorem \ref{thm:weaksols_delta}, this system can also we written in the weaker formulation (cf. \eqref{eq:varform:phi00}--\eqref{eq:varform:sigma00}):
        \begin{align}
            & \duality{\partial_t \phid, v}_V
            + (\vec{J}_\delta, \nabla v)_H
            \notag \\
    		& \quad = \bigg( P(\phid) \bigg( \sigmad + \chi \left( 1 - \left( \phid - \dg \mud \right) \right) + \gamma \Delta \left( \phid - \dg \mud \right) \notag \\
            & \qquad - \Psi_+'(\phid) - \Psi_-'\left( \phid - \dg \mud \right) + \chi \sigmad \bigg), v \bigg)_{\!\! H} \label{eq:varform:phidelta00} \\
            & (\vec{J}_\delta, \vec{\xi})_H
            = - \left(\gamma \Delta \left( \phid - \dg \mud \right) \, b'(\phid) \nabla \phid, \vec{\xi} \right)_{\!\! H}
            - \left( \gamma b(\phid) \Delta \left( \phid - \dg \mud \right), \div \vec{\xi} \right)_{\!\! H} \notag \\
            & \quad + (b(\phid) \Psi_+''(\phid) \nabla \phid, \vec{\xi})_H
            + \left( b(\phid) \Psi_-''\left(\phid - \dg \mud\right) \nabla \left( \phid - \dg \mud \right), \vec{\xi} \right)_{\!\! H} \notag \\
            & \qquad - \chi (b(\phid) \nabla \sigmad, \vec{\xi})_H
            \label{eq:varform:Jdelta} \\
            & \duality{\partial_t \sigmad, v}_V
            + (\nabla (\sigmad + \chi (1 - \phid)), \nabla v)_H \notag \\
            & \quad = - \bigg( P(\phid) \bigg( \sigmad + \chi \left( 1 - \left( \phid - \dg \mud \right) \right) + \gamma \Delta \left( \phid - \dg \mud \right) \notag \\
            & \qquad - \Psi_+'(\phid) - \Psi_-'\left( \phid - \dg \mud \right) + \chi \sigmad \bigg), v \bigg)_{\!\! H} \label{eq:varform:sigmadelta00}
        \end{align}
        for any $v \in V$, for any $\vec{\xi} \in H^1(\Omega; \R^d) \cap L^\infty(\Omega; \R^d)$ and almost everywhere in $(0,T)$.
        Our aim is to show that $(\phid, \mud, \sigmad)$ converges to a weak solution $(\phi, \vec{J}, \sigma)$ to problem \eqref{eq:phi00}--\eqref{eq:ic00}, in the sense of Definition \ref{def:weaksol_limit},
        as $\delta \to 0$ along a suitable subsequence.

        The main difficulty lies in handling the terms related to the flux $\vec{J}_\delta$.
        To do this, we  need one more uniform bound.
        Recall that for any fixed $\delta > 0$, the flux $\vec{J}_\delta$ is well-defined as $\vec{J}_\delta = b(\phid) \nabla \mud + b(\phid)\Psi_+''(\phid) \nabla \phid$ (see also \ref{ass:psi} and \ref{ass:b}).
        Moreover, from \eqref{eq:conv_Jeps1} and \eqref{eq:conv_Jeps2} we know that
        \[
            \vec{J}_{\delta, \eps} = b_\eps(\phide) \nabla \mude + b_\eps(\phide)\Psipe''(\phide) \nabla \phide \weak \vec{J}_\delta
            \quad \hbox{weakly in $\LT 2 H$ as $\eps \to 0$.}
        \]
        Thus, by the weak lower-semicontinuity of the norms, we get
        \begin{equation}
        \label{eq:bound_Jdelta}
        \begin{split}
            \norm{\vec{J}_\delta}^2_{\LT 2 H}
            & \le \liminf_{\eps \to 0} \,\, \norm{\vec{J}_{\delta, \eps}}^2_{\LT 2 H} \\
            & \le \liminf_{\eps \to 0} \,\, \norm{b_\eps}_\infty \int_\Omega b_\eps(\phide) \abs{\nabla \mude + \Psipe''(\phide) \nabla \phide}^2 \, \de x \le C,
        \end{split}
        \end{equation}
        by \eqref{eq:energyineq_eps} and the boundedness of $b$, where $C > 0$ is independent of $\delta$.

        We now collect all the uniform estimates and use known compactness results to extract converging subsequences.
        If not further specified, all the convergences below have to be intended as $\delta \to 0$ along a suitable subsequence.
        First of all, from (see \eqref{eq:bound_mu_l2v_delta})
        \[
            \sqrt{\delta} \norm{\mud}_{\LT 2 V} \le C,
        \]
        it follows that
        \begin{equation}
            \label{eq:conv_mu_l2v_delta}
            \delta \mud \to 0 \quad \hbox{strongly in $\LT 2 V$}.
        \end{equation}
        Then, by Banach--Alaoglu's theorem, the uniform bounds \eqref{eq:bound_sigma_l2v_delta}, \eqref{eq:bound_phidgmu_h2_delta}, and \eqref{eq:bound_phi_l2v_delta} imply that
        \begin{align}
            & \phid \weak \phi \quad \hbox{weakly in $\LT 2 V$,} \label{eq:conv_phi_l2v_delta} \\
            & \sigmad \weakstar \sigma \quad \hbox{weakly star in $\LT \infty H \cap \LT 2 V$,} \label{eq:conv_sigma_l2v_delta} \\
            & \phid - \dg \mud \weakstar \phi \quad \hbox{weakly star in $\LT \infty V \cap \LT 2 {\Hx2}$,} \label{eq:conv_phidgmu_l2w_delta}
        \end{align}
        where we also used \eqref{eq:conv_mu_l2v_delta} to identify the limit in \eqref{eq:conv_phidgmu_l2w_delta}.
        Similarly, by \eqref{eq:bounds_dt_delta}, we deduce the following weak convergences:
        \begin{align}
            & \phid \weak \phi \quad \hbox{weakly in $\HT 1 {V^*}$,} \label{eq:conv_phi_h1vs_delta} \\
            & \sigmad \weak \sigma \quad \hbox{weakly in $\HT 1 {V^*}$,} \label{eq:conv_sigma_h1vs_delta} \\
            & \phid - \dg \mud \weak \phi \quad \hbox{weakly in $\HT 1 {V^*}$.} \label{eq:conv_phidgmu_h1vs_delta}
        \end{align}
        Consequently, by known compact embeddings (see \cite[Section 8, Corollary 4]{S1986}), \eqref{eq:conv_phi_l2v_delta}--\eqref{eq:conv_phidgmu_l2w_delta} and \eqref{eq:conv_phi_h1vs_delta}--\eqref{eq:conv_phidgmu_h1vs_delta} imply the following strong convergences:
        \begin{align}
            & \phid \to \phi \quad \hbox{strongly in $\LT 2 H$,} \label{eq:conv_phi_l2h_delta} \\
            & \sigmad \to \sigma \quad \hbox{strongly in $\LT 2 H$,} \label{eq:conv_sigma_l2h_delta} \\
            & \phid - \dg \mud \to \phi \quad \hbox{strongly in $\LT 2 V$.} \label{eq:conv_phidgmu_l2v_delta}
        \end{align}
        In particular, \eqref{eq:conv_mu_l2v_delta} and \eqref{eq:conv_phidgmu_l2v_delta} also entail that
        \begin{equation}
            \label{eq:conv_phi_l2v_strong_delta}
            \phid \to \phi \quad \hbox{strongly in $\LT 2 V$.}
        \end{equation}
        Additionally, from \eqref{eq:conv_phi_l2h_delta}--\eqref{eq:conv_phi_l2v_strong_delta} we also infer that
        \begin{align}
            & \phid \to \phi, \, \nabla \phid \to \nabla \phi \quad \hbox{a.e.~in $Q_T$,} \label{eq:conv_phi_ae_delta} \\
            & \sigmad \to \sigma \quad \hbox{a.e.~in $Q_T$,} \label{eq:conv_sigma_ae_delta} \\
            & \phid - \dg \mud \to \phi \quad \hbox{a.e.~in $Q_T$.} \label{eq:conv_phidgmu_ae_delta}
        \end{align}
        From \eqref{eq:conv_phi_ae_delta}, since $0 \le \phid \le 1$ by Theorem \ref{thm:weaksols_delta}, we deduce that
        \[
            0 \le \phi \le 1 \quad \hbox{a.e.~in $Q_T$.}
        \]
       Recalling \eqref{eq:muprime}, thanks to \eqref{eq:conv_phidgmu_l2w_delta}, \eqref{eq:conv_phidgmu_ae_delta}, the continuity and the linear growth of $\Psi'_-$, the Dominated Convergence Theorem and \eqref{eq:conv_sigma_l2v_delta}, we deduce that
        \begin{equation}
            \label{eq:conv_mu_l2h_delta}
            \mud \weak - \gamma \Delta \phi + \Psi'_-(\phi) - \chi \sigma \quad \hbox{weakly in $\LT 2 H$.}
        \end{equation}
       Therefore, by \eqref{eq:bound_Jdelta}, it also follows that
        \begin{equation}
            \label{eq:conv_J_delta}
            \vec{J}_\delta \weak \vec{J} \quad \hbox{weakly in $\LT 2 H$.}
        \end{equation}
        Then, the limit variables $(\phi, \vec{J}, \sigma)$ satisfy the regularity properties stated in Theorem \ref{thm:conv_deltato0}.
        
        We are now left to check that $(\phi, \vec{J}, \sigma)$ satisfy the identities \eqref{eq:varform:phi00}--\eqref{eq:varform:sigma00}.
        Starting from \eqref{eq:varform:phidelta00}--\eqref{eq:varform:sigmadelta00}, the weak convergences above easily imply the convergence of all the linear terms.
        Then, as before, we are left with the nonlinear terms.
        We start from the reaction terms appearing on the right-hand sides of \eqref{eq:varform:phidelta00} and \eqref{eq:varform:sigmadelta00}.
        First of all, by the continuity and boundedness of $P$ and \eqref{eq:conv_phi_ae_delta}, the Dominated Convergence Theorem implies that
        \begin{equation}
            \label{eq:conv_P_delta}
            P(\phid) \to P(\phi) \quad \hbox{strongly in $\LT p {\Lx p}$ for any $p \ge 1$.}
        \end{equation}
        Hence, the weak convergences in $\LT 2 H$ of $\sigmad$, $\phid - \dg \mud$ and $\Delta \left( \phid - \dg \mud \right)$ by \eqref{eq:conv_sigma_l2v_delta} and \eqref{eq:conv_phidgmu_l2w_delta} are enough to conclude that
        \begin{align*}
            & \int_\Omega P(\phid) \left( (1 + \chi) \sigmad + \chi \left( 1 - \left(\phid - \dg \mud \right) \right) + \gamma \Delta \left(\phid - \dg \mud \right) \right) v \, \de x \\
            & \quad \to
            \int_\Omega P(\phi) \left( (1 + \chi) \sigma + \chi (1 - \phi) - \gamma \Delta \phi \right) v \, \de x,
        \end{align*}
        since $v \in V \hookrightarrow \Lx 6$ and \eqref{eq:conv_P_delta} can be applied with $p=3$.
        Then, we must handle the two terms involving the convex splitting of $\Psi$.
        For the convex one, we use \ref{ass:P} together with  \eqref{eq:conv_phi_ae_delta}, the fact that $0 \le \phid, \phi \le 1$ and the Dominated Convergence Theorem, to conclude that
        \[
            P(\phid) \Psi_+'(\phid) \to P(\phi) \Psi_+'(\phi) \quad \hbox{strongly in $\LT 2 H$.}
        \]
        For the concave one, we use again \eqref{eq:conv_phi_ae_delta}, \eqref{eq:conv_phidgmu_ae_delta}, the continuity and boundedness of $P$, the continuity and linear growth of $\Psi_-'$ and the Dominated Convergence Theorem, to obtain
        \[
            P(\phid) \Psi_-'\left(\phid - \dg \mud \right) \to P(\phi) \Psi_-'(\phi) \quad \hbox{strongly in $\LT 2 H$.}
        \]
        This provides the convergence of the full reaction term.
        Our last task is now to identify the limit the right-hand side of \eqref{eq:varform:Jdelta}.
        We proceed term by term.
        For the first one, we observe that, by \ref{ass:b_c1} and \eqref{eq:conv_phi_ae_delta}, we have
        \[
            b'(\phid) \nabla \phid \to b'(\phi) \nabla \phi \quad \hbox{a.e.~in $Q_T$.}
        \]
        Moreover, by the boundedness of $b$ and \eqref{eq:conv_phi_l2v_strong_delta}, we have that
        \[
            \abs{b'(\phid) \nabla \phid} \le \norm{b'}_\infty \abs{\nabla \phid} \to \norm{b'}_\infty \abs{\nabla \phi} \quad \hbox{strongly in $\LT 2 H$.}
        \]
        Then, by the generalised Dominated Convergence Theorem, it follows that
        \begin{equation}
            \label{eq:conv_bphi_delta}
            b'(\phid) \nabla \phid \to b'(\phi) \nabla \phi \quad \hbox{strongly in $\LT 2 H$.}
        \end{equation}
        Consequently, \eqref{eq:conv_phidgmu_l2w_delta},  \eqref{eq:conv_bphi_delta}, and the crucial fact that $\vec{\xi} \in L^\infty(\Omega; \R^d)$ imply that
        \[
            - \int_\Omega \gamma \Delta \left( \phid - \dg \mud \right) b'(\phid) \nabla \phid \cdot \vec{\xi} \, \de x
            \to
            - \int_\Omega \gamma \Delta \phi \, b(\phi) \nabla \phi \cdot \vec{\xi} \, \de x.
        \]
        For the second term, we observe that, by the continuity and boundedness of $b$ and \eqref{eq:conv_phi_ae_delta}, the Dominated Convergence Theorem gives
        \begin{equation}
            \label{eq:conv_b_delta}
            b(\phid) \to b(\phi) \quad \hbox{strongly in $\LT 2 H$.}
        \end{equation}
        This convergence combined with \eqref{eq:conv_phidgmu_l2w_delta} and the uniform bound
        \[
            \norm*{b(\phid) \Delta \left(\phid - \dg \mud \right)}_{\LT 2 H} \le \norm{b}_\infty \norm*{\Delta \left(\phid - \dg \mud \right)}_{\LT 2 H} \le C,
        \]
        which follows from \eqref{eq:unifbounds_entropy_delta}, imply that
        \[
            b(\phid) \Delta \left(\phid - \dg \mud \right) \weak b(\phi) \Delta \phi \quad \hbox{weakly in $\LT 2 H$.}
        \]
        Then, we proved the convergence of the second term, since $\div \vec{\xi} \in H$.
        Next, we recall that $b \Psi_+'' \in C^0([0,1])$ by \ref{ass:b}. Thus,, by \eqref{eq:conv_phi_ae_delta} and the Dominated Convergence Theorem, we deduce that
        \[
            b(\phid) \Psi_+''(\phid) \to b(\phi) \Psi_+''(\phi) \quad \hbox{strongly in $\LT 2 H$,}
        \]
        which in combination with \eqref{eq:conv_phi_l2v_strong_delta} yields the convergence of the third term.
        Similarly, the pointwise convergences \eqref{eq:conv_phi_ae_delta} and \eqref{eq:conv_phidgmu_ae_delta}, and the boundedness and continuity of $b$ and $\Psi_-''$, using again the Dominated Convergence Theorem, imply that
        \[
            b(\phid) \Psi_-''\left( \phid - \dg \mud \right) \to b(\phi) \Psi_-''(\phi) \quad \hbox{strongly in $\LT 2 H$.}
        \]
        This convergence and \eqref{eq:conv_phidgmu_l2v_delta} allow us to pass in the fourth term.
        Finally, \eqref{eq:conv_sigma_l2v_delta} and \eqref{eq:conv_b_delta} also yield the convergence of the last term on the right-hand side of \eqref{eq:varform:Jdelta}.
        The proof of Theorem \ref{thm:conv_deltato0} is finished.
   
\appendix \section{A technical lemma}
Here we prove a result which has been used in the previous proof (see \eqref{singpert}).
    \begin{lemma}
    \label{lemma:singpert}
        For any $\delta > 0$, let $u_\delta \in V$ be the unique weak solution to the following elliptic problem
        \begin{alignat}{2}
            & - \delta \Delta u_\delta + u_\delta = f + \delta g + \delta h(u_\delta) \quad && \hbox{in $\Omega$,} \label{eq:elliptic} \\
            & \partial_{\n} u_\delta = 0 \quad && \hbox{on $\partial \Omega$,} \label{eq:elliptic_bc}
        \end{alignat}
        where $f \in V$, $g \in H$ and $h \in C^1(\R)$ such that for any $s \in R$
        \begin{align}
            & \abs{h(s)} \le k_1 \abs{s} + k_2, \label{eq:elliptic_h1} \\
            & \abs{h'(s)} \le k_3, \label{eq:elliptic_h2}
        \end{align}
        for some $k_1, k_2, k_3 > 0$. Then, as $\delta$ goes to $0$ along a suitable sequence, we have
        \[
            u_\delta \to f \quad \hbox{strongly in $V$.}
        \]
    \end{lemma}

    \begin{proof}
        First of all, we observe that, by standard results on elliptic equations, the weak solution $u_\delta \in V$ exists and it is unique under the given hypotheses.
        Moreover, by elliptic regularity theory, since $f  \in V$, $g \in H$ and $h$ is Lipschitz, it also follows that $u_\delta \in \Hx 2$ is actually a strong solution to \eqref{eq:elliptic}--\eqref{eq:elliptic_bc}.

        Testing \eqref{eq:elliptic} with $u_\delta$ and using the Cauchy--Schwarz and the Young inequalities, together with \eqref{eq:elliptic_h1}, we get that
        \begin{align*}
            \delta \norm{\nabla u_\delta}^2_H + \norm{u_\delta}^2_H
            & \le \mezzo \norm{u_\delta}^2_H + \mezzo \norm{f}^2_H + \frac{\delta}{2} \norm{u_\delta}^2_H + \frac{\delta}{2} \norm{g}^2_H + \delta \int_\Omega h(u_\delta) u_\delta \, \de x \\
            & \le \mezzo \norm{u_\delta}^2_H + \mezzo \norm{f}^2_H + \frac{\delta}{2} \norm{u_\delta}^2_H + \frac{\delta}{2} \norm{g}^2_H \\
            & \qquad + \delta k_1 \norm{u_\delta}^2_H + \delta \frac{k_2}{2} \norm{u_\delta}^2_H + \delta \frac{k_2}{2} \abs{\Omega}.
        \end{align*}
        Then, collecting the terms and multiplying by $2$, we obtain
        \begin{equation}
            \label{eq:elliptic:est1}
            \delta \norm{\nabla u_\delta}^2_H + (1 - \delta (1 + 2 k_1 + k_2)) \norm{u_\delta}^2_H
            \le \norm{f}^2_H + \delta \norm{g}^2_H + \delta k_2 \abs{\Omega}.
        \end{equation}
        Next, we test \eqref{eq:elliptic} with $- \Delta u$ and, by \eqref{eq:elliptic_h2}, we infer that
        \begin{align*}
            \delta \norm{\Delta u_\delta}^2_H + \norm{\nabla u_\delta}^2_H
            & \le \mezzo \norm{\nabla u_\delta}^2_H + \mezzo \norm{\nabla f}^2_H + \frac{\delta}{2} \norm{\Delta u_\delta}^2_H + \frac{\delta}{2} \norm{g}^2_H + \delta \int_\Omega h'(u_\delta) \nabla u_\delta \cdot u_\delta \, \de x \\
            & \le \mezzo \norm{\nabla u_\delta}^2_H + \mezzo \norm{\nabla f}^2_H + \frac{\delta}{2} \norm{\Delta u_\delta}^2_H + \frac{\delta}{2} \norm{g}^2_H + \delta k_3 \norm{\Delta u_\delta}^2_H.
        \end{align*}
        Consequently, we also have the following inequality:
        \begin{equation}
            \label{eq:elliptic:est2}
            \delta \norm{\Delta u_\delta}^2_H + (1 - 2 k_3 \delta) \norm{\nabla u_\delta}^2_H
            \le \norm{\nabla f}^2_H + \delta \norm{g}^2_H.
        \end{equation}
        Now, let us set
        \[
            c_\delta := \min\{ 1 - \delta (1 + 2k_1 + k_2), 1 - 2k_3 \delta \},
        \]
        and observe that $c_\delta > 0$ if $\delta$ is small enough and $c_\delta \to 1$ as $\delta \to 0$.
        Then, by neglecting the first terms in the respective inequalities, the combination of \eqref{eq:elliptic:est1} and \eqref{eq:elliptic:est2} implies that
        \begin{equation}
            \label{eq:elliptic:est3}
            \norm{u_\delta}^2_V \le \frac{1}{c_\delta} \norm{f}^2_V + \frac{\delta}{c_\delta} \norm{g}^2_H + \frac{\delta}{c_\delta} \frac{k_2 \abs{\Omega}}{2},
        \end{equation}
        where the right-hand side is uniformly bounded for small $\delta$.
        Hence, by Banach--Alaoglu's Theorem, we can extract a non-relabelled subsequence such that
        \[
            u_\delta \weak \bar{u} \quad \hbox{weakly in $V$ as $\delta \to 0$,}
        \]
        for some $\bar{u} \in V$.
        Additionally, from the weak formulation of \eqref{eq:elliptic} for any $v \in V$, one easily deduces that the weak convergence above implies that $\bar{u} = f$ in $V$.
       In order to show the strong convergence in $V$, it suffices to establish the convergence of the norms.
        This follows from the following chain of inequalities, relying on the weak lower-semicontinuity of the norm, \eqref{eq:elliptic:est3} and the fact that $c_\delta \to 1$ as $\delta \to 0$.
        Indeed, we have that
        \begin{align*}
            \norm{f}^2_V
            & \le \liminf_{\delta \to 0} \, \norm{u_\delta}^2_V
            \le \limsup_{\delta \to 0} \, \norm{u_\delta}^2_V \\
            & \le \lim_{\delta \to 0} \left( \frac{1}{c_\delta} \norm{f}^2_V + \frac{\delta}{c_\delta} \norm{g}^2_H + \frac{\delta}{c_\delta} \frac{c_2 \abs{\Omega}}{2} \right)
            = \norm{f}^2_V.
        \end{align*}
        Then, $\norm{u_\delta}_V \to \norm{f}_V$, and, together with the weak convergence above, this proves Lemma \ref{lemma:singpert}.
    \end{proof}

\noindent
{\bf Acknowledgments.} 
The authors wish to thank the anonymous reviewers, who carefully read the manuscript and provided many comments that improved the quality of the paper.
This work originated during a visit of C.~Cavaterra and M.~Grasselli to the Laboratoire ``Jacques-Louis Lions'', Sorbonne Universit\'e, Paris, whose hospitality is gratefully acknowledged. C.~Cavaterra, M.~Fornoni, and M.~Grasselli are members of GNAMPA (Gruppo Nazionale per l'Ana\-li\-si Matematica, la Probabilit\`{a} e le loro Applicazioni) of INdAM (Istituto Nazionale di Alta Matematica).
This research is part of the activities of ``Dipartimento di Eccellenza 2023-2027'' of Universit\`a degli Studi di Milano (C.~Cavaterra and M. Fornoni) and Politecnico di Milano (M.~Grasselli).



\end{document}